\documentclass[12pt]{amsart}
\usepackage{amsmath,amsthm,amssymb}
\usepackage{enumerate}
\usepackage{graphicx}
\usepackage{float}
\usepackage{epstopdf}
\usepackage{cite}
\usepackage[margin=1in]{geometry}
\usepackage{tikz}
\usetikzlibrary{matrix}
\usetikzlibrary{arrows}
\usetikzlibrary{positioning}
\usepackage{overpic}
\usepackage{mathtools}
\usepackage{subcaption}

\theoremstyle{plain}
\newtheorem{thm}{Theorem}[section]
\newtheorem{cor}[thm]{Corollary}
\newtheorem{lem}[thm]{Lemma}
\newtheorem{prop}[thm]{Proposition}
\newtheorem{conj}[thm]{Conjecture}
\theoremstyle{definition}
\newtheorem{defn}[thm]{Definition}
\theoremstyle{remark}
\newtheorem{rem}[thm]{Remark}
\newtheorem{ex}[thm]{Example}
\newtheorem*{notation}{Notation}
\newcommand{\Z}{\mathbb{Z}}
\newcommand{\Q}{\mathbb{Q}}
\newcommand{\hfk}{\mathit{HFK}}
\newcommand{\cfk}{\mathit{CFK}}
\newcommand{\cH}{\mathcal{H}}
\newcommand{\sln}{\mathfrak{sl}_{n}}
\DeclareMathOperator{\unknot}{unknot}
\DeclareMathOperator{\gr}{gr}
\DeclareMathOperator{\rank}{rank}
\DeclareMathOperator{\alg}{alg}
\DeclareMathOperator{\spin}{Spin}

\begin{document}

\title{A Family of $\mathfrak{sl}_{n}$-like Invariants in knot Floer homology}

\author{Nathan Dowlin}
\thanks{The author was partially supported by NSF grant DMS-1606421}
\maketitle

\begin{abstract} We define and study a family of link invariants $\hfk_{n}(L)$. Although these homology theories are defined using holomorphic disc counts, they share many properties with $\sln$ homology. Using these theories, we give a framework that generalizes the conjectured spectral sequence from Khovanov homology to $\delta$-graded knot Floer homology. In particular, we conjecture that for all links $L $ in $ S^3$ and all $n\ge 1$, there is a spectral sequence from the $\sln$ homology of $L$ to $\hfk_{n}(L)$.
\end{abstract}

\tableofcontents
\normalsize

\vspace{-5mm}

\section{Introduction} One of the most active topics of research in knot theory is understanding the relationships between the quantum invariants of Khovanov \cite{Khov1} and Khovanov-Rozansky \cite{KR}, \cite{KR2} and the Heegaard Floer invariants of Ozsv\'{a}th-Szab\'{o} \cite{OS1} and Rasmussen \cite{Rasmussen2}. There are two main conjectures on this topic:

\begin{conj}[\hspace{1sp}\cite{Rasmussen3}]

For any knot $K$ in $S^{3}$, there is a spectral sequence from $\overline{Kh}(K)$ to $\delta$-graded $\widehat{\hfk}(K)$, where $\overline{Kh}(K)$ is the reduced Khovanov homology of $K$ and $\widehat{\hfk}(K)$ is the reduced knot Floer homology of $K$.

\end{conj}

\begin{conj}[\hspace{1sp}\cite{Gukov}]

For any link $L$ in $S^{3}$, there are spectral sequences from $\overline{H}(L)$ to $\widehat{\hfk}(L)$ and from $H(L)$ to $\hfk(L)$, where $\overline{H}(L)$ is the reduced HOMFLY-PT homology of $L$, $H(L)$ is the unreduced HOMFLY-PT homology of $L$, and $\hfk(L)$ is the unreduced knot Floer homology of $L$.

\end{conj}

These conjectures are somewhat incomplete, as there are more invariants on the quantum side for which there are no analogs on the knot Floer side - in particular, there are no analogs of the $\sln$ homology of Khovanov and Rozansky. In this paper we define a family of knot homology theories $\hfk_{n}(L)$ where $L$ is a null-homologous link in a 3-manifold $Y$ which fill in this gap. There are also reduced versions $\widehat{\hfk}_{n}(L)$, with $\widehat{\hfk}_{2}(K)$ isomorphic to $\delta$-graded $\widehat{\hfk}(K)$ for any knot $K$.

Since $\mathfrak{sl}_{2}$ homology is isomorphic to Khovanov homology and HOMFLY-PT homology can be viewed as a ``limit" of $\sln$ homology, these theories allow us to interpolate between the two conjectures. HOMFLY-PT and $\sln$ homology are only defined for knots in $S^{3}$, so we hope that these invariants will help with the long-term project of generalizing $\sln$ and HOMFLY-PT homology to links in more general 3-manifolds.

The knot Floer complex comes in many versions, but the most general is the one which counts discs which pass through all basepoints in the Heegaard diagram, each with their own variable. We will refer to this curved complex as the  \emph{master complex} (this complex is studied in detail by Zemke in \cite{Zemke2}). The invariants $\hfk_{n}(L)$ are defined using the master complex with two additional ingredients:

\vspace{1mm}

$\bullet$ The Heegaard diagram comes with a puncture $p$ and an additional $\alpha$ and $\beta$ curve.

$\bullet$ We make some identifications in the ground ring.

\vspace{1mm}

We will describe the master complex and the construction of $\hfk_{n}(L)$ in the next section, but in the case where $L$ is a knot and the Heegaard diagram has a unique $w$ and $z$ basepoint, both are particularly easy to describe. Given a punctured Heegaard diagram $\cH$ with a single pair of basepoints, the master complex $\cfk_{U,V}(\cH)$ is defined over $\Q[U,V]$ with differential

\[  \partial_{U,V}(x) = \sum_{y \in \mathbb{T_{\alpha}} \cap \mathbb{T_{\beta}}} \sum_{\substack{\phi \in \pi_{2}(x,y) \\ \mu(\phi)=1 \\ n_{p}(\phi)=0}}    \# \widehat{\mathcal{M}}(\phi) U^{n_{w}(\phi)}V^{n_{z}(\phi)} y \]

\noindent
where $p$ is the puncture on the Heegaard diagram. The identification in the ground ring to obtain $\cfk_{n}(K)$ is given by $V=nU^{n-1}$. Thus, the complex $\cfk_{n}(K)$ is a free $\Q[U]$-module with differential

\[  \partial_{n}(x) = \sum_{y \in \mathbb{T_{\alpha}} \cap \mathbb{T_{\beta}}} \sum_{\substack{\phi \in \pi_{2}(x,y) \\ \mu(\phi)=1 \\ n_{p}(\phi)=0}}    \# \widehat{\mathcal{M}}(\phi) U^{n_{w}(\phi)}(nU^{n-1})^{n_{z}(\phi)} y \]

\noindent
and $\hfk_{n}(K)$ is the homology of this complex. The reduced theories $\widehat{\hfk}(L)$ are obtained by removing the puncture in the Heegaard diagram and setting $U=0$.

\begin{rem}

When the Heegaard diagram has a unique pair of basepoints $(w,z)$, the master complex is a true complex. However, for multi-pointed Heegaard diagrams, it is a curved complex with $\partial_{U,V}^{2}=\omega I$. In this more general case, the identifications in the ground ring are necessary for $\partial_{n}^2=0$. After these identifications, the differential is no longer homogeneous with respect to the Maslov and Alexander gradings, but it is homogeneous with respect to a linear combination of the two given by $\gr_{n} = -nM + 2(n-1)A $.

\end{rem}

\begin{thm}

The homology theories $\hfk_{n}(L)$ and $\widehat{\hfk}_{n}(L)$ are graded link invariants. 
\end{thm}

\noindent
\textbf{Properties of $\hfk_{n}$}:

\vspace{2mm}

a) $\hfk_{n}( \unknot)=\Q[U]/U^{n}=0$.

\vspace{.75mm}
b) If $L $ is a link in $ Y$, $\hfk_{1}(L) \cong \widehat{HF}(Y)$.

\vspace{1.25mm}
c) If $L_{1}$, $L_{2}$ are links in $Y_{1}, Y_{2}$, then 
\[\hfk_{n}( L_{1} \sqcup L_{2}) \cong \hfk_{n}(L_{1}) \otimes \hfk_{n}(L_{2}) \]

\hspace{5mm}where $L_{1} \sqcup L_{2}$ is a link in $Y_{1}\#Y_{2}$.

\vspace{1mm}
d) For all $l$-component links $L$, $\hfk_{n}(L)$ is a finitely generated module over \[\Q[U_{1},...,U_{l}]/(U_{1}^{n}=...=U_{l}^{n}=0)\] 

e) For all links $L$ and all $n\ge1$, there is a spectral sequence from $\hfk(L)$ to $\hfk_{n}(L)$.

\vspace{3mm}

\noindent
All of these properties are also satisfied by the $\sln$ homology of Khovanov and Rozansky \cite{KR}, where the last property corresponds to Rasmussen's spectral sequences from HOMFLY-PT homology to $\sln$ homology \cite{Rasmussen}.

When $Y=S^{3}$, $\hfk_{n}$ readily extends to singular links. Let $S$ be a planar diagram for a singular link in $S^3$. We say that $S$ is \emph{completely singular} if it has no crossings. When $S$ is completely singular, the definition of $\hfk_{n}(S)$ agrees with the one given in \cite{Me2}, where we showed the following:

\begin{thm}[\hspace{1sp}\cite{Me2}]

Let $S$ be a completely singular braid diagram. There is an isomorphism of graded vector spaces 
\[  \hfk_{n}(S) \cong H_{n}(S)  \]

\noindent
where $H_{n}(S)$ is the unreduced $\sln$ homology of $S$.

\end{thm}

This theorem inspires the following conjectured relationship between $\hfk_{n}(L)$ and $H_{n}(L)$, where $L$ is a link in $S^{3}$.

\begin{conj}

Following the conventions of \cite{Rasmussen}, let $\mathbf{gr}_{n}, \mathbf{gr}_{v}$ be the bigrading on $\sln$ homology. For each $n$, there are spectral sequences \[H_{n}(L) \to \hfk_{n}(L) \hspace{2mm}\text{ and }\hspace{2mm} \overline{H}_{n}(L) \to \widehat{\hfk}_{n}(L) \]

\vspace{1mm}
\noindent
where the grading on $H_{n}(L)$ is given by $\mathbf{gr}_{n}+\frac{n}{2}\mathbf{gr}_{v}$. Each differential in the spectral sequence decreases this grading by $n$.

\end{conj}

Since $\widehat{\hfk}_{2}(K)$ is isomorphic to $\delta$-graded $\widehat{\hfk}(K)$, this gives a generalization of Rasmussen's conjectured spectral sequence from Khovanov homology to $\delta$-graded knot Floer homology. One interesting aspect of these theories is that there is no difficulty in extending to links with multiple components. Historically, this has been a fairly significant obstacle in relating Khovanov homology to knot Floer homology and was the inspiration behind the definitions for pointed links given in \cite{BLS}. These spectral sequences, together with Rasmussen's spectral sequences and the $\hfk_{n}$ analogs, are illustrated diagrammatically in Figure \ref{specpic}. In Section \ref{section3} we show that this conjecture would follow from our theory satisfying an oriented skein exact triangle condition (Conjecture \ref{conj3.9}).

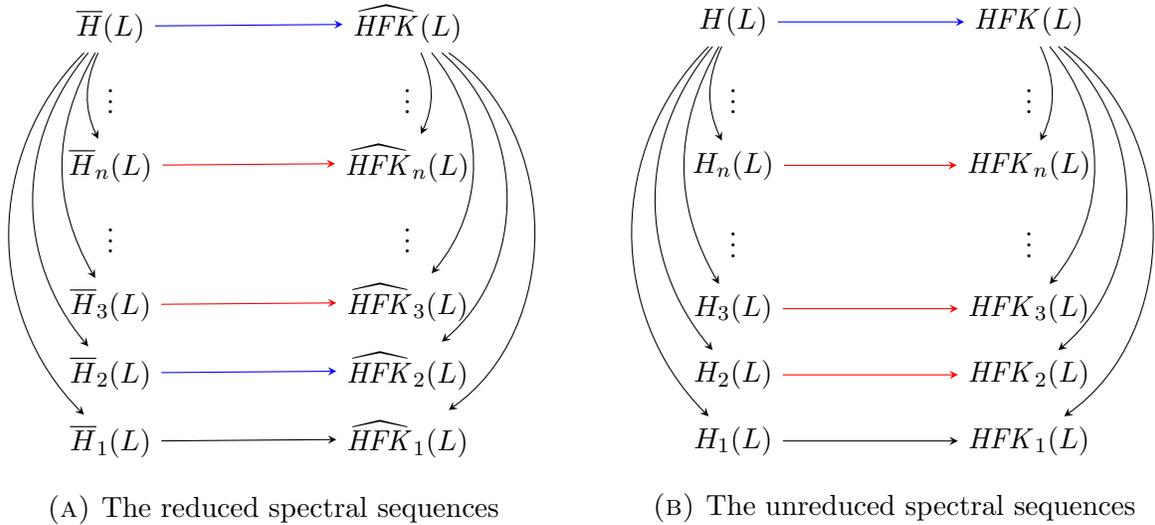
\begin{figure} \small
\centering
\begin{subfigure}{.5\textwidth}
  \centering
\begin{tikzpicture}
  \matrix (m) [matrix of math nodes,row sep=.4em,column sep=6em,minimum width=2em] {
     \overline{H}(L) &  \widehat{\hfk}(L)  \\
    \vdots & \vdots \\
    \overline{H}_{n}(L) & \widehat{\hfk}_{n}(L)  \\
    \vdots & \vdots \\
    \overline{H}_{3}(L) & \widehat{\hfk}_{3}(L)\\
    \overline{H}_{2}(L) & \widehat{\hfk}_{2}(L)\\
    \overline{H}_{1}(L) & \widehat{\hfk}_{1}(L)\\};
  \path[-stealth]
	(m-1-1) edge [blue] node {} (m-1-2)
	(m-1-1) edge [bend right = 25] node{} (m-3-1)
	(m-1-1) edge [bend right = 30] node{} (m-5-1)
	(m-1-1) edge [bend right = 40] node{} (m-6-1)
	(m-1-1) edge [bend right = 45] node{} (m-7-1)
	(m-1-2) edge [bend left = 25] node{} (m-3-2)
	(m-1-2) edge [bend left = 40] node{} (m-5-2)
	(m-1-2) edge [bend left = 50] node{} (m-6-2)
	(m-1-2) edge [bend left = 55] node{} (m-7-2)
	(m-3-1) edge [red] node {} (m-3-2)
	(m-5-1) edge [red] node {} (m-5-2)
	(m-6-1) edge [blue] node {} (m-6-2)
	(m-7-1) edge node {} (m-7-2);
\end{tikzpicture}
  \caption{The reduced spectral sequences}
\end{subfigure}%
\begin{subfigure}{.5\textwidth}
  \centering
\begin{tikzpicture}
  \matrix (m) [matrix of math nodes,row sep=.62em,column sep=6em,minimum width=2em] {
     H(L) &  \hfk(L)  \\
    \vdots & \vdots \\
    H_{n}(L) & \hfk_{n}(L)  \\
    \vdots & \vdots \\
    H_{3}(L) & \hfk_{3}(L)\\
    H_{2}(L) & \hfk_{2}(L)\\
    H_{1}(L) & \hfk_{1}(L)\\};
  \path[-stealth]
	(m-1-1) edge [blue] node {} (m-1-2)
	(m-1-1) edge [bend right = 25] node{} (m-3-1)
	(m-1-1) edge [bend right = 30] node{} (m-5-1)
	(m-1-1) edge [bend right = 40] node{} (m-6-1)
	(m-1-1) edge [bend right = 45] node{} (m-7-1)
	(m-1-2) edge [bend left = 25] node{} (m-3-2)
	(m-1-2) edge [bend left = 40] node{} (m-5-2)
	(m-1-2) edge [bend left = 50] node{} (m-6-2)
	(m-1-2) edge [bend left = 55] node{} (m-7-2)
	(m-3-1) edge [red] node {} (m-3-2)
	(m-5-1) edge [red] node {} (m-5-2)
	(m-6-1) edge [red] node {} (m-6-2)
	(m-7-1) edge node {} (m-7-2);
\end{tikzpicture}
  \caption{The unreduced spectral sequences}
\end{subfigure}
\caption{A diagram which includes both the proved and conjectured spectral sequences. The black arrows correspond to proved spectral sequences, the blue arrows to previously conjectured spectral sequences, and the red arrows to the open conjectures which are new to this paper.} \label{specpic}
\end{figure}

It is not hard to show that for $n=1$ these spectral sequences exist - in fact 
\[ \overline{H}_{1}(L) = \widehat{\hfk}_{1}(L) = \Q \]
\[ H_{1}(L) = \hfk_{1}(L) = \Q \]

\noindent
The first interesting example is $n=2$. This spectral sequence has already been explored on the reduced side, but the unreduced version is a new conjecture. Since much more is known about Khovanov homology than $\sln$ homology for $n \ge 3$, this gives a case where we can prove the conjecture for several classes of knots.

\begin{thm}

There is a spectral sequence from $H_{2}(K)$ to $\hfk_{2}(K)$ if $K$ is a 2-bridge knot, the $(3,m)$ torus knot, or any knot with crossing number $\le 13$.

\end{thm}

We also have some results for $n \ge 3$, but since much less is known about $\sln$ homology, we need to rely on a conjectured formula.

\begin{prop}

Let $K$ be a 2 or 3-strand torus knot, and assume the conjectured formula for $\sln$ homology of 3-strand torus knots given in \cite{GorskyLewark} holds true. Then there is a spectral sequence from $H_{n}(K)$ to $\hfk_{n}(K)$.

\end{prop}

\section{Defining the Invariants $\hfk_{n}(L)$} 

\subsection{Heegaard Diagrams} We give a brief overview of Heegaard diagrams for knots and links, and the adaptation to punctured Heegaard diagrams.

\begin{defn}

A (multipointed) Heegaard diagram $\mathcal{H}$ for a link $L$ in a 3-manifold $Y$ is a tuple $(\Sigma, \alpha, \beta, \bf{w}, \bf{z})$ where 

$\bullet$ $\Sigma$ is a genus $g$ surface.

$\bullet$ $\alpha$ (resp. $\beta$) is a set of disjoint embedded circles $\alpha_{1},..., \alpha_{g+k-1}$ (resp. $\beta_{1},...,\beta_{g+k-1}$) such that the $\{\alpha_{i}\}$ and $\{ \beta_{i} \}$ intersect transversely.

$\bullet$ The $\alpha$ and $\beta$ each span a $g$ dimensional subspace of $H_{1}(\Sigma)$.

$\bullet$ $\bf{w}$ (resp. $\bf{z})$ is a set of basepoints $w_{1},...,w_{k}$ (resp. $z_{1},...,z_{k}$) such that each component of $\Sigma  \backslash \alpha$ contains both a $\bf{w}$ basepoint and a $\bf{z}$ basepoint, and similarly for each component of $\Sigma  \backslash \beta$.

\end{defn}

The data $(\Sigma, \alpha, \beta)$ determine a 3-maniflold $Y$ as follows: start with  $\Sigma \times [0,1]$, and glue in thickened discs along the $\alpha$ curves on $\Sigma \times \{0\}$. The resulting boundary (on that side) consists of $k$ disjoint 2-spheres, which we can fill in with 3-balls. Repeat this process with the $\beta$ curves on the $\Sigma \times \{1 \}$ boundary. The result is a closed, oriented 3-manifold $Y$.

Similarly, the basepoints $\bf{w}, \bf{z}$ determine a link $L$ in $Y$. Start with the arcs $(\bf{w} \sqcup \bf{z})\times [0,1]$. Each component of $\Sigma  \backslash \alpha$ contains a basepoint $w_{i}$ and $z_{j}$, which we connect with an arc in the ball filling in the $\alpha$ boundary. Similarly, for each component of $\Sigma  \backslash \beta$, we connect the unique $w_{i}$ and $z_{j}$ with an arc in the ball filling in the $\beta$ boundary. The link $L$ is oriented from the $\bf{w}$ basepoints to the $\bf{z}$ basepoints on the $\alpha$ side and from the $\bf{z}$ to the $\bf{w}$ basepoints on the $\beta$ side.

\begin{notation}

We will write the link $(Y,L)$ simply as $L$ when it won't cause confusion.

\end{notation}

An important piece of data that we'll need to keep track of for this construction is which two $w_{i}$ are connected to each $z_{i}$ basepoint. Let $w_{a(i)}$ be the basepoint connected to $z_{i}$ on the $\alpha$ side and $w_{b(i)}$ the basepoint connected to $z_{i}$ on the $\beta$ side.

\begin{defn}

A \emph{punctured} Heegaard diagram for a knot or link is a standard Heegaard diagram with three extra pieces of data: $p, \alpha_{g+k}$, and $\beta_{g+k}$. The point $p \in \Sigma$ can be viewed as a puncture, and the extra $\alpha$ an $\beta$ circles are required to separate the $\bf{w}$ and $\bf{z}$ basepoints from $p$. The $\alpha$ and $\beta$ circles are still required to be transverse, and each must still span a $g$-dimensional subspace of $H_{1}(\Sigma)$.

\end{defn}

\begin{rem} \label{remark1} 

A punctured Heegaard diagram for  can be viewed as a standard Heegaard diagram for a the link $L\sqcup (\unknot)$ for which the extra unknotted component contains only the basepoints $w_{0}$ and $z_{0}$, and $w_{0}=z_{0}$. The point $p$ is recording the location of these two basepoints. This puncture comes up in \cite{BLS} as well when they are moving from a reduced to an unreduced theory.

\end{rem}

\subsection{Knot Floer Chain Complexes} All of our complexes will be defined over $\Q$, so we will need to keep track of signs but not torsion. Suppose $\cH$ is a (unpunctured) Heegaard diagram for a null-homologous link $L$ in $Y$. The standard knot Floer complex $\cfk^{-}(\cH)$ is defined over $\Q[U_{1},...,U_{k}]$ with differential

\[  \partial^{-}(x) = \sum_{y \in \mathbb{T_{\alpha}} \cap \mathbb{T_{\beta}}} \sum_{\substack{\phi \in \pi_{2}(x,y) \\ \mu(\phi)=1 \\ n_{\bf{z}}(\phi)=0}}    \# \widehat{\mathcal{M}}(\phi) U_{1}^{n_{w_{1}}(\phi)}\cdot \cdot \cdot U_{k}^{n_{w_{k}}(\phi)} y \]

In order to define the differential with signs, we need to choose a trivialization of the determinant line bundle $det(\phi)$ for each class $\phi$ subject to certain compatibility conditions. Such a choice determines an orientation on each moduli space $\widehat{\mathcal{M}}(\phi)$ allowing us to count the discs with signs, and the compatibility conditions guarantee that $(\partial^{-})^{2}=0$. All of these choices together make up a \emph{system of orientations}. It was shown in \cite{Akram} that one can always find a system of orientations in which the $\alpha$-degenerations come with positive sign and the $\beta$-degenerations come with negative sign - we will always work with such a system of orientations. Any two systems of orientations with this property are strongly equivalent, so they have the same chain homotopy type \cite{SucharitSign}.

It was shown by Oszv\'{a}th and Szab\'{o} that for any link $L$, the chain homotopy type of $\cfk^{-}(\cH)$ does not depend on the choice of Heegaard diagram for $L$, nor on the system of orientations. The homology of this complex is denoted $\hfk^{-}(L)$ and by a minor abuse of notation, the chain complex is denoted $\cfk^{-}(L)$.

The complex comes equipped with two gradings, Maslov and Alexander. In this paper we will only define them as relative gradings, though absolute versions can be pinned down with some extra work. If $\phi \in \pi_{2}(x,y)$, then
\[ M(x) - M(y) = \mu(\phi) - 2n_{\bf{w}}(\phi) \hspace{3mm} \text{ and }\hspace{3mm}A(x) - A(y) = n_{\bf{z}}(\phi) - n_{\bf{w}}(\phi) \]

\noindent
Multiplication by any of the $U_{i}$ lowers Maslov grading by 2 and lowers Alexander grading by 1. It follows from these definitions that the differential $\partial^{-}$ drops Maslov grading by 1 and preserves Alexander grading.

There is also a more general complex called the \emph{master complex}, which we will denote $\cfk_{U,V}(\cH)$. The ground ring is $\Q[U_{1},...U_{k},V_{1},...,V_{k}]$, and the differential is given by 

\[  \partial_{U,V}(x) = \sum_{y \in \mathbb{T_{\alpha}} \cap \mathbb{T_{\beta}}} \sum_{\substack{\phi \in \pi_{2}(x,y) \\ \mu(\phi)=1}}    \# \widehat{\mathcal{M}}(\phi) U_{1}^{n_{w_{1}}(\phi)}\cdot \cdot \cdot U_{k}^{n_{w_{k}}(\phi)}V_{1}^{n_{z_{1}}(\phi)} \cdot \cdot \cdot V_{k}^{n_{z_{k}}(\phi)} y \]

\noindent
Multiplication by any of the $V_{i}$ preserves Maslov grading but raises Alexander grading by 1. Unfortunately $\partial_{U,V}^{2} \ne 0$, but it can be computed by counting the Maslov index 2 degenerations.

\begin{lem} The map $\partial_{U,V}: \cfk_{U,V}(\cH) \to \cfk_{U,V}(\cH)$ satisfies
\begin{equation} \label{d2}
\partial_{U,V}^{2} = \sum_{i=1}^{k} (U_{a(i)}-U_{b(i)})V_{i}     
\end{equation}
\end{lem}

\begin{proof}

Each component of $\Sigma \backslash \alpha$ gives a degeneration with positive sign and each component of $\Sigma \backslash \beta$ gives a degeneration with negative sign. The component of $\Sigma \backslash \alpha$ containing $z_{i}$ also contains $w_{a(i)}$, so this $\alpha$-degeneration has contribution $U_{a(i)}V_{i}$. Similarly, the component of $\Sigma \backslash \beta$ containing $z_{i}$ also contains $w_{b(i)}$, so this $\beta$-degeneration has contribution $-U_{b(i)}V_{i}$, giving the formula.

\end{proof}

The same theory can be defined for punctured Heegaard diagrams, with the restriction that we don't count discs which pass through the puncture $p$:

\[  \partial_{U,V}(x) = \sum_{y \in \mathbb{T_{\alpha}} \cap \mathbb{T_{\beta}}} \sum_{\substack{\phi \in \pi_{2}(x,y) \\ \mu(\phi)=1 \\ n_{p}(\phi)=0}}    \# \widehat{\mathcal{M}}(\phi) U_{1}^{n_{w_{1}}(\phi)}\cdot \cdot \cdot U_{k}^{n_{w_{k}}(\phi)}V_{1}^{n_{z_{1}}(\phi)} \cdot \cdot \cdot V_{k}^{n_{z_{k}}(\phi)} y \]

\noindent
Since the region containing the puncture $p$ is blocked, $\partial_{U,V}^{2}$ is the same as in the unpunctured case. The relative gradings are given by 
\[ M(x) - M(y) = \mu(\phi) - 2n_{\bf{w}}(\phi) - 2n_{p}(\phi) \hspace{3mm}  \text{ and }\hspace{3mm}A(x) - A(y) = n_{\bf{z}}(\phi) - n_{\bf{w}}(\phi) \]

\noindent
Let $\cH$ and $\cH'$ be a punctured an an unpunctured Heegaard diagram, respectively, for a link $L$. We define the absolute gradings on $\cfk^{-}(\cH)$ so that $\hfk^{-}(\cH) \cong \hfk^{-} \otimes V$, where $V=\Q\{-1,-1/2\} \oplus \Q\{0,-1/2\}$. These gradings determine an absolute bigrading on $\cfk_{U,V}(\cH)$.

\subsection{The Definition of $\hfk_{n}(L)$} Suppose $\cH$ is a punctured Heegaard diagram for $L$. In order to make $\cfk_{U,V}(\cH)$ into an (untwisted) chain complex, some identifications need to be made in the ground ring. The standard solution is to set $V_{i}=V_{j}$ if $z_{i}$ and $z_{j}$ lie on the same component of $L$. Instead, we are going to draw on ideas from the $\sln$ homology of Khovanov and Rozansky \cite{KR}. 

\begin{defn}

Define the chain complex $\cfk_{n}(\cH)$ to be the quotient 
\begin{equation}\label{substitute}
\cfk_{n}(\cH) = \cfk_{U,V}(\cH)/\{ V_{i} = (U_{a(i)}^{n} - U_{b(i)}^{n})/(U_{a(i)}-U_{b(i)}) \} 
\end{equation}

\noindent
The induced differential on this complex will be denoted $\partial_{n}$.

\end{defn}

\begin{lem}

The map $\partial_{n}: \cfk_{n}(\cH) \to \cfk_{n}(\cH)$ satisfies $\partial_{n}^{2}=0$.

\end{lem}

\begin{proof}

Plugging $V_{i} = (U_{a(i)}^{n} - U_{b(i)}^{n})/(U_{a(i)}-U_{b(i)})$ into (\ref{d2}) gives 

\[ \partial_{n}^{2} = \sum_{i=1}^{k} U_{a(i)}^{n} - U_{b(i)}^{n}    \]

\noindent
since $a$ and $b$ are both permutations, this sums to zero.

\end{proof}

Let $\hfk_{n}(\cH) = H_{*}(\cfk_{n}(\cH), \partial_{n})$. The differential $\partial_{n}$ is not homogeneous with respect to the Maslov and Alexander gradings, but if we define the grading $\gr_{n} = -nM + 2(n-1)A$, then $\partial_{n}$ is homogeneous of degree $n$ with respect to $\gr_{n}$ and each $U_{i}$ has grading $2$.

Suppose $L$ is an $l$-component link. We will choose an ordering on the $w_{i}$ such that each of $w_{1},...,w_{l}$ lie on different components of $L$.

\begin{thm} \label{inv}

If $\cH_{1}$ and $\cH_{2}$ are two punctured Heegaard diagrams for a link $L$, then $\cfk_{n}(\cH_{1})$ and $ \cfk_{n}(\cH_{2})$ are chain homotopy equivalent as graded $\Q[U_{1},...,U_{l}]$-modules.

\end{thm}

\noindent
This theorem states that our homology theory is a link invariant. We will give a proof in the following section. As part of the proof, we will see that if $w_{i}$ and $w_{j}$ lie on the same component of $L$, then they have chain homotopic actions on $\cfk_{n}(\cH)$ (see Corollary \ref{cor1}). 
\begin{defn}

Define $\hfk_{n}(L) = H_{*}(\cfk_{n}(\cH),\partial_{n})$ where $\cH$ is any punctured Heegaard diagram for $L$.

\end{defn}

By a minor abuse of notation, we will also write $\cfk_{n}(L)$ for the chain complex coming from any punctured Heegaard diagram for $L$.

\begin{rem}

Each generator $x \in \mathbb{T}_{\alpha} \cap \mathbb{T}_{\beta}$ can be assigned a $\spin^{c}$ structure $\mathfrak{s} \in \spin^{c}(Y)$ \cite{3manifold}. Since $L$ is null-homologous in $Y$, the differential $\partial_{n}$ preserves the underlying $\spin^{c}$ structure, so the resulting homology splits

\[   \hfk_{n}(Y,L) = \bigoplus_{\mathfrak{s} \in \spin^{c}(Y)} \hfk_{n}(Y,L,\mathfrak{s})  \]

\noindent
However, since we are focusing on the relationships of $\hfk_{n}$ with $\sln$ homology, our examples will only be in $S^{3}$ which has a unique $\spin^{c}$ structure, so we will suppress $\spin^{c}$ structures from the notation.

\end{rem}

Now that we have a well-defined homology theory, let's compute it for a few knots. If $R$ is a bigraded polynomial ring with gradings $M$ and $A$, let $R\{i,j\}$ be the free $R$-module on one generator with the gradings shifted such that $1$ lies in Maslov grading $i$ and Alexander grading $j$. Similarly, if $R$ is singly graded with grading $\gr_{n}$, let $R\{k\}$ be the ring with gradings shifted such that $\gr_{n}(1)=k$.

\begin{ex}[The Unknot] Consider the punctured Heegaard diagram for the unknot in Figure \ref{UnknotHD}. There are two generators $x,y \in \mathbb{T}_{\alpha} \cap \mathbb{T}_{\beta}$, and the associated chain complex $\cfk_{U,V}(\cH)$ is given by 
\[ \Q[U,V]\{-1,-1/2 \} \xrightarrow{\hspace{3mm} UV \hspace{3mm}} \Q[U,V]\{0,-1/2\}  \]

\noindent
To get the complex $\cfk_{n}(\unknot)$, we must substitute for $V$ as in Equation (\ref{substitute}). Since there is only one $w$ basepoint, $z$ is connected to it on both the $\alpha$ and $\beta$ sides, so the formula becomes 
\[ V = \frac{d}{dU}U^{n} = n U^{n-1} \]

\noindent
Thus, the complex $\cfk_{n}(\unknot)$ is given by 
\[ \Q[U]\{0\} \xrightarrow{\hspace{3mm} nU^{n} \hspace{3mm}} \Q[U]\{-n+1\}  \]

\noindent
so the homology is $\hfk_{n}(\unknot) = \Q[U]/U^{n}=0$ with $\gr_{n}(1)=-n+1$. Note that the graded Euler characteristic of this homology is $ \frac{q^{n}-q^{-n}}{q-q^{-1}}$, which is the $\sln$ polynomial of the unknot.

\end{ex}

\begin{figure}[ht]
\centering
\def\svgwidth{6.5cm} \large
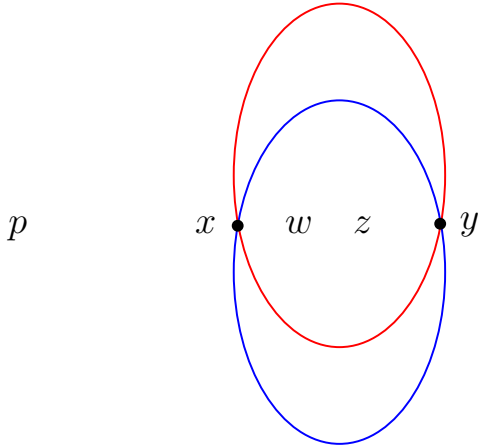
\caption{A Punctured Heegaard Diagram on $S^{2}$ for the Unknot}\label{UnknotHD}
\end{figure}

\begin{notation}

Let $[n]$ denote the homology $\hfk_{n}(\unknot)$.

\end{notation}

\begin{ex}[The right handed trefoil] Consider the punctured Heegaard diagram for the right-handed trefoil $T_{2,3}$ in Figure \ref{TrefoilHD}. The associated chain complex $\cfk_{U,V}(T_{2,3})$ is shown in Figure \ref{cfkuvtref}, and the complex $\cfk_{n}(T_{2,3})$ is shown in Figure \ref{cfkntref}. The homology of this complex is 
\begin{align*} 
\hfk_{n}(T_{2,3}) &= \Q[U]\{-n+3\}/(U^{n}=0) \oplus \Q\{n-1\} \oplus \Q\{-2n-1\} \\
&=[n]\{2\} \oplus [1]\{n-1\} \oplus [1]\{2n-1\} 
\end{align*}

\begin{figure}[ht]
\centering
\def\svgwidth{10cm} \small
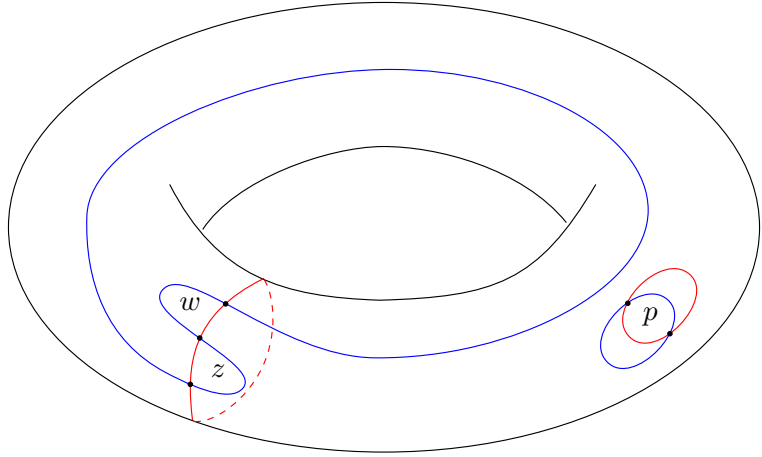
\caption{A punctured Heegaard diagram for the right handed trefoil.}\label{TrefoilHD}
\end{figure}

\begin{figure}[!h]
\centering
\begin{tikzpicture}
  \matrix (m) [matrix of math nodes,row sep=4.5em,column sep=5.5em,minimum width=2em] {
     \Q[U,V]\{-1,1/2\} & \Q[U,V]\{-2,-1/2\} & \Q[U,V]\{-3,-3/2\} \\
     \Q[U,V]\{0,1/2\} & \Q[U,V]\{-1,-1/2\} & \Q[U,V]\{-2,-3/2\} \\};
  \path[-stealth]
    (m-1-1) edge node [left] {$UV$} (m-2-1)
    (m-1-2) edge node [above] {$U$} (m-1-1)
    (m-1-2) edge node [above] {$V$} (m-1-3)
    (m-1-3) edge node [left] {$UV$} (m-2-3)
    (m-1-2) edge node [right] {$UV$} (m-2-2)
    (m-2-2) edge node [above] {$V$} (m-2-3)
    (m-2-2) edge node [above] {$U$} (m-2-1);
\end{tikzpicture}
\caption{The complex $\cfk_{U,V}(T_{2,3})$.} \label{cfkuvtref}
\end{figure}

\begin{figure}[!h]
\centering
\begin{tikzpicture}
  \matrix (m) [matrix of math nodes,row sep=4.5em,column sep=5.5em,minimum width=2em] {
     \Q[U]\{2n-1\} & \Q[U]\{n+1\} & \Q[U]\{3\} \\
     \Q[U]\{n-1\} & \Q[U]\{1\} & \Q[U]\{-n+3\} \\};
  \path[-stealth]
    (m-1-1) edge node [left] {$nU^{n}$} (m-2-1)
    (m-1-2) edge node [above] {$U$} (m-1-1)
    (m-1-2) edge node [above] {$nU^{n-1}$} (m-1-3)
    (m-1-3) edge node [left] {$nU^{n}$} (m-2-3)
    (m-1-2) edge node [right] {$nU^{n}$} (m-2-2)
    (m-2-2) edge node [above] {$nU^{n-1}$} (m-2-3)
    (m-2-2) edge node [above] {$U$} (m-2-1);
\end{tikzpicture}
\caption{The complex $\cfk_{n}(T_{2,3})$.} \label{cfkntref}
\end{figure}

\end{ex}

\newpage

\subsection{Invariance} In this section we will prove Theorem \ref{inv}. As usual, this will be proved via invariance under Heegaard moves.

\begin{lem}

Let $\cH_{1}$ and $\cH_{2}$ be two punctured Heegaard diagrams for a link $L$. Then $\cH_{1}$ and $\cH_{2}$ can be connected by a sequence of the following Heegaard moves: 

$\bullet$ Change of almost complex structure

$\bullet$ Isotopies which do not cross the basepoints $\bf{w},\bf{z}$ or the puncture $p$

$\bullet$ Handleslides

$\bullet$ (1,2)-stabilization (increasing the genus of $\Sigma$)

$\bullet$ (0,3)-stabilization (increasing the number of basepoints)

\end{lem}

\begin{proof}

This lemma essentially follows from the same fact being true for unpunctured Heegaard diagrams. Let $\cH_{1}'$ and $\cH_{2}'$ be the unpunctured Heegaard diagrams for $L \sqcup (\unknot)$ obtained by replacing $p$ with two basepoints $w_{0}$ and $z_{0}$ The diagrams $\cH_{1}'$ and $\cH_{2}'$ can be connected by a sequence of the above moves, so the same sequence of moves connects the punctured diagrams $\cH_{1}$ and $\cH_{2}$.

\end{proof}

We will now give some preliminary results on how the complex $\cfk_{U,V}$ is known to change under Heegaard moves.

\begin{lem} [\hspace{1sp}\cite{Zemke2}] \label{firstmoves}

If $\cH_{1}$ and $\cH_{2}$ are two Heegaard diagrams which differ by a choice of almost complex structure, isotopy, handleslide, or (1,2)-stabilization, then $\cfk_{U,V}(\cH_{1})$ and $\cfk_{U,V}(\cH_{2})$ are homotopy equivalent in the category of curved complexes.

\end{lem}

The only remaining move is (0,3)-stabilization. Let $\cH_{1}$ be a Heegaard diagram, and let $w_{a_{1}(i)}, z_{i}$ be two basepoints in $\cH_{1}$. Define $\cH_{2}$ to be the diagram obtained by performing (0,3)-stabilization between $z_{i}$ and $w_{a_{1}(i)}$ as in Figure \ref{stabdiagrams}. If $X$ is the set of generators for $\cfk_{U,V}(\cH_{1})$, then the set of generators for $\cfk_{U,V}(\cH_{1})$ is given by $\{x,y\} \times X$.

\begin{figure}
\centering
\begin{subfigure}{.4\textwidth}
  \centering
\def\svgwidth{7cm} \small
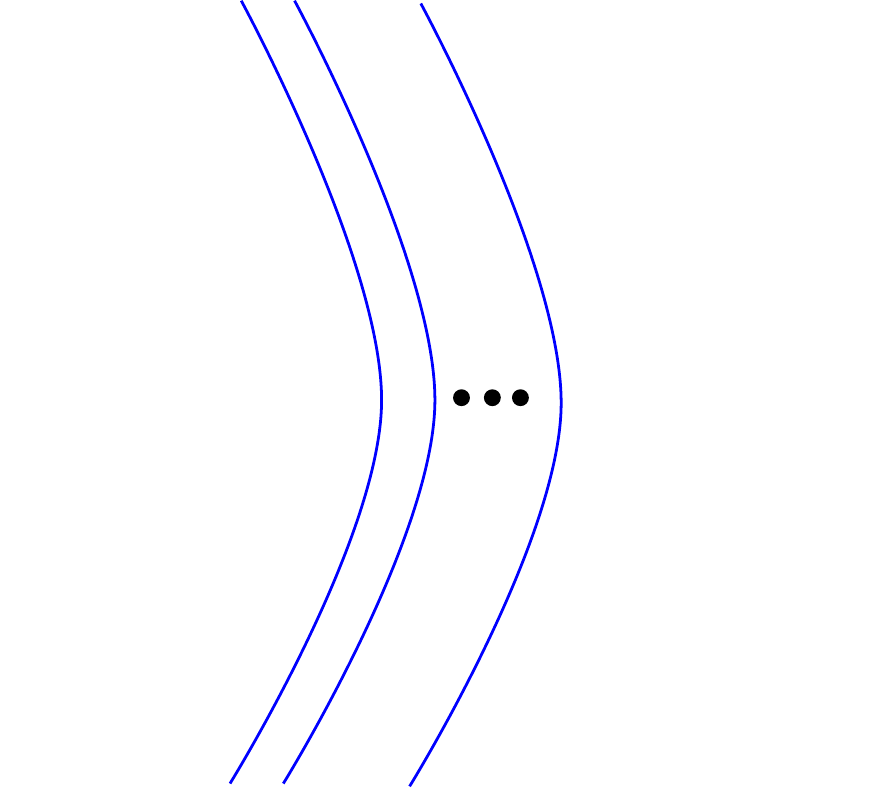
  \caption{The unstabilized diagram $\cH_{1}$}
\end{subfigure}%
\begin{subfigure}{.8\textwidth}
  \centering
\def\svgwidth{11cm} \small
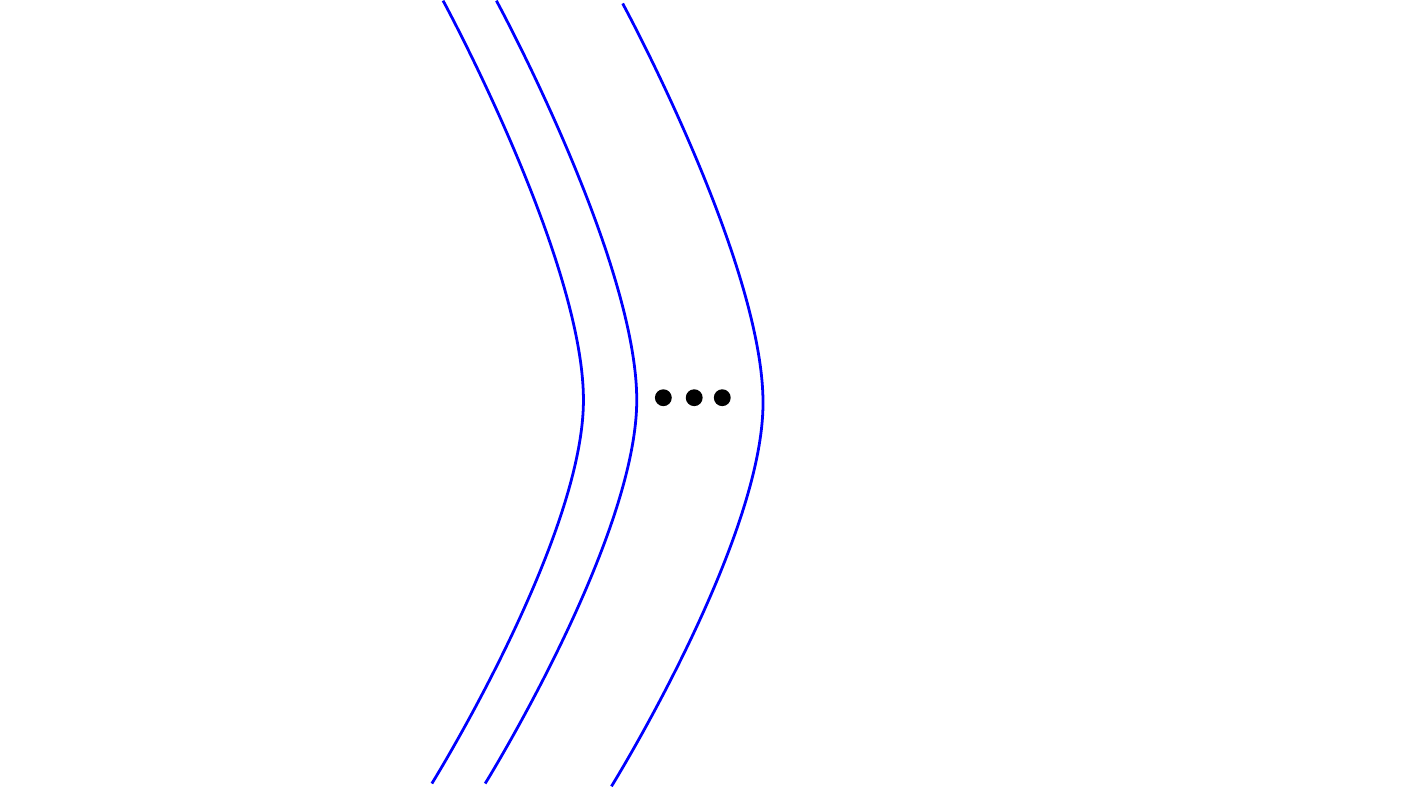
  \caption{The (0,3)-stabilized diagram $\cH_{2}$}
\end{subfigure}
\caption{A Heegaard diagram before and after (0,3)-stabilization}\label{stabdiagrams}
\end{figure}

\begin{lem}[\hspace{1sp}\cite{Akram}, \cite{Zemke1}] \label{30stab}

Let $\partial_{U,V,1}$ denote the differential on $\cfk_{U,V}(\cH_{1})$ and $\partial_{U,V,2}$ the differential on $\cfk_{U,V}(\cH_{2})$. For a suitable choice of complex structure, 

\begin{equation} \label{eq2}
 \partial_{U,V,2} = \begin{bmatrix}
    \partial_{U,V,1}      & V_{k+1}-V_{i}  \\
    U_{a_{1}(i)}-U_{k+1}       & \partial_{U,V,1} \\
\end{bmatrix} = \begin{bmatrix}
    \partial_{U,V,1}      & V_{k+1}-V_{i}  \\
    U_{a_{2}(k+1)}-U_{b_{2}(k+1)}       & \partial_{U,V,1} \\
\end{bmatrix} 
\end{equation}

\end{lem}

This formula is proved over $\Z$ by Alishahi and Eftekhary in \cite{Akram}, although they set $V_{k+1}=V_{i}$ to make $\partial^{2}=0$. However, repeating their argument without performing this quotient yields the formula above. It was also proved over $\Z_{2}$ by Zemke in \cite{Zemke1}, where he proves naturality of (0,3)-stabilization.

\begin{cor} \label{cor1}

Let $\cH$ be a Heegaard diagram for $L$. For any $1 \le i \le k$ there is a homotopy $H_{i}: \cfk_{U,V}(\cH) \to \cfk_{U,V}(\cH)$ such that 

\[ \partial_{U,V}H_{i} + H_{i} \partial_{U,V} = U_{a(i)} - U_{b(i)} \]

\end{cor}

\begin{proof}

In the special case where $a(i) = b(i)$, take $H_{i}=0$. Otherwise the diagram $\cH$ can be changed to a diagram $\cH'$ which looks like a (0,3) stabilization around $z_{i}$ via isotopies and handlesides. By the previous theorem, for a suitable choice of complex structure the differential $\partial_{U,V}'$ can be written as a two by two matrix as in (\ref{eq2}). Taking 
\[ H'_{i} = \begin{bmatrix}
    0      & 0 \\
    1       & 0\\
\end{bmatrix} \]

\noindent
we get $\partial_{U,V}'H_{i}' + H_{i}'\partial_{U,V}' = U_{a(i)} - U_{b(i)}$. Since $\cfk_{U,V}(\cH)$ is homotopy equivalent to $\cfk_{U,V}(\cH')$, this proves the corollary.

\end{proof}

\begin{proof}[Proof of Theorem \ref{inv}]

Most of the work has already been done for us. The previous lemmas regarding $\cfk_{U,V}$ of normal Heegaard diagrams immediately extend to punctured Heegaard diagrams viewing the puncture as an unknotted component of the link with basepoints $w_{0}$ and $z_{0}$, and then setting $U_{0}=V_{0}=0$ in the ground ring.

 Plugging $V_{j} = (U_{a(j)}^{n} - U_{b(j)}^{n})/(U_{a(j)}-U_{b(j)})$ into Lemma \ref{firstmoves} shows that the chain homotopy type of the complex $\cfk_{n}$ is invariant under change of complex structure maps, isotopies, handleslides, and (1,2)-stabilization. The only remaining Heegaard move to check it for is (0,3)-stabilization. 

Let $\cH_{1}$ and $\cH_{2}$ be two punctured Heegaard diagrams which differ by (0,3)-stabilization as in Lemma \ref{30stab}. The two $\cfk_{U,V}$ complexes are related as follows:

\begin{figure}[!h]
\centering
\begin{tikzpicture}
  \node at (-6,0) {$\cfk_{U,V}(\cH_{2})=$};
  \matrix (m) [matrix of math nodes,row sep=5em,column sep=8em,minimum width=2em] {
     \cfk_{U,V}(\cH_{1})& \cfk_{U,V}(\cH_{1}) \\};
  \path[-stealth]
    (m-1-1) edge [bend left=15] node [above] {$U_{a_{1}(i)}-U_{k+1}$} (m-1-2)
    (m-1-2) edge [bend left=15] node [below] {$V_{k+1}-V_{i}$} (m-1-1);
\end{tikzpicture}
\end{figure}

\noindent
There is some subtlety here in passing from $\cfk_{U,V}$ to $\cfk_{n}$. In $\cH_{1}$, the basepoint $z_{i}$ is connected to $w_{a_{1}(i)}$ on the $\alpha$ side, but in $\cH_{2}$ it is connected to $w_{k+1}$, so equation (\ref{substitute}) will be substituting different values in the two cases. Let $\cfk'_{n}(\cH_{1})$ denote the quotient in which $V_{j}=(U_{a_{1}(j)}^{n}-U_{b(j)}^{n})/(U_{a_{1}(j)}-U_{b(j)})$ for $j \ne i$, and $V_{i}=(U_{k+1}^{n}-U_{b(j)}^{n})/(U_{k+1}-U_{b(i)})$. Then

\begin{figure}[!h]
\centering
\begin{tikzpicture}
  \node at (-6.5,0) {$\cfk_{n}(\cH_{2})=$};
  \matrix (m) [matrix of math nodes,row sep=5em,column sep=8em,minimum width=2em] {
     \cfk'_{n}(\cH_{1})[U_{k+1}]& \cfk'_{n}(\cH_{1})[U_{k+1}] \\};
  \path[-stealth]
    (m-1-1) edge [bend left=15] node [above] {$U_{a_{1}(i)}-U_{k+1}$} (m-1-2)
    (m-1-2) edge [bend left=15] node [below] {$V_{k+1}-V_{i}$} (m-1-1);
\end{tikzpicture}
\end{figure}

\noindent
where $V_{k+1}-V_{i}$ is now equivalent to a degree $n-1$ polynomial in the $U_{j}$. This complex is chain homotopy equivalent to $ \cfk'_{n}(\cH_{1})[U_{k+1}]/(U_{k+1}=U_{a_{1}(i)})$ over $\Q[U_{1},...,U_{k}]$. But $U_{k+1}=U_{a_{1}(i)}$ implies $V_{i}=(U_{a_{1}(i)}^{n}-U_{b(i)}^{n})/(U_{a_{1}(j)}-U_{b(i)})$, which is its value in $\cfk_{n}(\cH_{1})$. Thus, $\cfk_{n}(\cH_{2})$ is chain homotopy equivalent to $\cfk_{n}(\cH_{2})$ over $\Q[U_{1},...,U_{k}]$.

We have shown invariance of the chain homotopy type of $\cfk_{n}$ under all Heegaard moves. However, whenever we perform a (0,3)-destabilization, we lose one variable in our ground ring. Since we can destabilize the Heegaard diagram for $L$ until there is only one pair of basepoints on each component, the chain homotopy type is only an invariant as a module over $\Q[U_{1},...,U_{l}]$, where $l$ is the number of components of $L$ and $w_{1},...,w_{l}$ all lie on different components of $L$.

\end{proof}

\subsection{An Alternate Definition for Knots} \label{alt} Our computation of $\hfk_{n}(T_{2,3})$ demonstrates a method that can be used for any knot $K$. Let $\cH$ be an unpunctured Heegaard diagram for $K$ with a single pair of basepoints $w$ and $z$. Then the master complex $\cfk_{U,V}(\cH)$ is true chain complex, i.e. $\partial_{U,V}^{2}=0$. In this case, the homology $\hfk_{n}(K)$ can be computed directly from $\cfk_{U,V}(\cH)$ without passing to a punctured Heegaard diagram. This is particularly useful as the master complex has already been computed for many knots, including torus knots \cite{OS4}.

\begin{lem} \label{puncturelemma2}

If $\cH$ is an unpunctured Heegaard diagram for $K$ with a single pair of basepoints $w,z$, then 
\[ \hfk_{n}(K) = H_{*}(\cfk_{U,V}(\cH)/\{UV=0, V=nU^{n-1} \})\]

\end{lem}

\begin{proof}

Let $\cH'$ be the punctured diagram obtained from $\cH$ by adding $p, \alpha_{g+k}, \beta_{g+k}$ as in Figure \ref{Puncture2} away from any curves or basepoints on $\cH$. The analysis for this diagram is the same as in the (0,3)-stabilization diagram - the only difference is the location of basepoints. 
\begin{figure}[ht]
\centering
\def\svgwidth{4cm} 
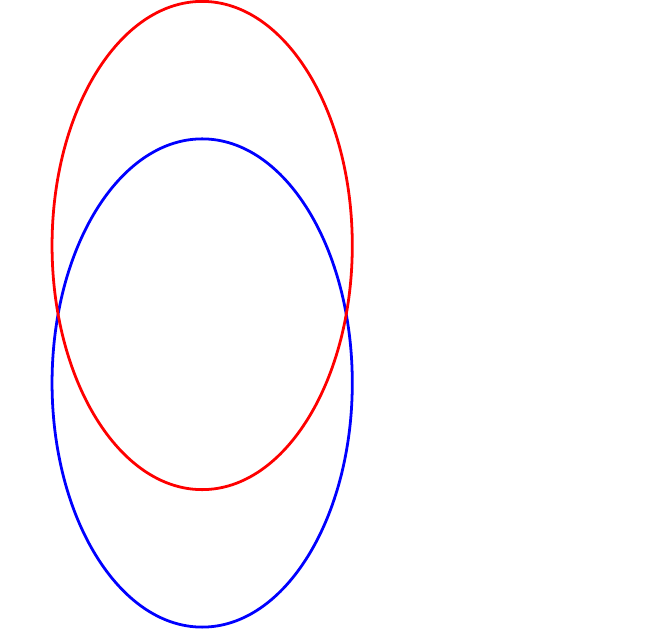
\caption{Adding a puncture to $\cH$}\label{Puncture2}
\end{figure}

For a suitable choice of complex structure, there are two homotopy classes of discs from $y$ to $x$ which contribute coefficient $1$ and $-1$, respectively, and two homotopy classes of discs from $x$ to $y$. One has $n_{p}(\phi)=1$, so it doesn't contribute. The other has $n_{w}(\phi)=n_{z}(\phi)=1$, so it comes with coefficient $UV$. Thus, $\cfk_{U,V}(\cH')$ is given by 
\[ \cfk_{U,V}(\cH)\{-1,0\} \xrightarrow{\hspace{3mm}UV\hspace{3mm}} \cfk_{U,V}(\cH)\{0,0\}  \]

\noindent
Canceling this arrow yields $\cfk_{U,V}(\cH)\{0,0\}/(UV=0)$. Substituting $V=nU^{n-1}$ yields the lemma.

\end{proof}

\begin{cor}

For any knot $K$, $\hfk_{n}(K)$ is a finitely generated $\Q[U]/(U^{n}=0)$ module. In particular, $\hfk_{n}(K)$ is finite dimensional over $\Q$.

\end{cor}

\subsection{The Reduced Homology Theories $\widehat{\hfk}_{n}(L)$} \label{reducedsection} In HOMFLY-PT homology, the reduced theory is not obtained from the unreduced theory by setting one of the variables equal to zero. Instead, one has to pass through an intermediate theory called the middle HOMFLY-PT homology. Unreduced HOMFLY-PT homology is obtained from the middle theory by tensoring with a two-dimensional vector space, and reduced HOMFLY-PT homology is obtained from the middle theory by setting one of the variables equal to zero.

We will be defining the reduced versions of $\hfk_{n}$ in an analogous way. Let $L$ be a link in $Y$, and let $\cH$ be a punctured Heegaard diagram for $L$. Define $\cfk(L)=\cfk^{-}(\cH)$, and let $\hfk(L)$ be the homology of this complex.

\begin{lem} The homology $\hfk(L)$ is given by 
\[ \hfk(L) \cong \hfk^{-}(L) \otimes V \]

\noindent
where $V=\Q\{-1,-1/2\} \oplus \Q\{0,-1/2\}$.

\end{lem}

\begin{proof}
This follows from interpreting the punctured Heegaard diagram as an unpunctured Heegaard diagram for $L \sqcup (\unknot)$ together with our grading convention for punctured Heegaard diagrams. 

\end{proof}

Thus, in our analogy with HOMFLY-PT homology, $\hfk(L)$ corresponds to the unreduced theory and $\hfk^{-}(L)$ corresponds to the middle theory. For this reason, we define the reduced theory as follows.

\begin{defn}

Let $L$ be a decorated link, i.e. a link with a marked component. Let $\cH$ be an unpunctured Heegaard diagram for $L$, with $w_{i}$ one of the basepoints on the marked component. Define the reduced complex $\widehat{\cfk}_{n}(L)$ by 
\[ \widehat{\cfk}_{n}(L) = \cfk_{n}(\cH)/U_{i}=0 \]

\noindent
The reduced homology $\widehat{\hfk}_{n}(L)$ is the homology of this complex.

\end{defn}

Note that this definition depends on a choice of component of the link $L$. The fact that it does not depend on the choice of basepoint $w_{i}$ on the component follows from Corollary \ref{cor1}.

\begin{lem}

Let $K$ be a knot, and let $n\ge 2$. Then $\widehat{\hfk}_{n}(K) \cong \widehat{\hfk}(K)$, viewing $\widehat{\hfk}(K)$ as singly graded with grading $\gr_{n}$.

\end{lem}

\begin{proof}

Let $\cH$ be a Heegaard diagram for $K$ with a single $w$ and $z$ basepoint. Then $V=nU^{n-1}$, so for $n \ge 2$, setting $U=0$ also forces $V=0$. Thus, $\widehat{\cfk}_{n}(\cH) = \widehat{\cfk}(\cH)$ if we view $\widehat{\cfk}(\cH)$ as singly graded with grading $\gr_{n}$.

\end{proof}

In the special case where $n=1$, each $V_{i}=1$ and $\gr_{n}$ is just the Maslov grading, so $\widehat{\hfk}_{1}(L) \cong \widehat{\mathit{HF}}(Y) $.

\subsection{Properties of $\hfk_{n}(L)$} The homology theories $\hfk_{n}$ share many properties with $\sln$ homology, which we will describe in this section. One of the most interesting is the fact that both are finite-dimensional over $\Q$.

\begin{lem}

Let $L$ be an $l$-component link. Then $\hfk_{n}(L)$ is a finitely-generated module over $\Q[U_{1},...,U_{l}]/(U_{1}^{n}=...=U_{l}^{n}=0)$, where the $w_{i}$ are ordered such that $w_{1},...,w_{l}$ each lie on a different component of $L$.

\end{lem}

\begin{proof} From Corollary \ref{cor1}, we know that $\hfk_{n}(L)$ is a finitely-generated module over $\Q[U_{1},...,U_{l}]$. It remains to show that $U_{i}^{n}=0$. 

Let $\cH$ be an unpunctured Heegaard diagram for $L$ with a single pair of basepoints $w_{i},z_{i}$ for each component. We add a puncture to $\cH$ as in Figure \ref{Puncture2} in the region containing the basepoint $w_{i}$. Following the proof of Lemma \ref{puncturelemma2}, there is a unique holomorphic disc from $y$ to $x$ which does not pass through $p$, which comes with coefficient $U_{i}V_{i}$. Thus, $\cfk_{n}(L)$ is given by
\[ \cfk_{U,V}(\cH) \xrightarrow{\hspace{3mm}U_{i}V_{i}\hspace{3mm}} \cfk_{U,V}(\cH)\  \]

\noindent
Substituting $V_{j}=nU_{j}^{n-1}$ for all $j$, this becomes 
\[ \cfk_{n}(\cH) \xrightarrow{\hspace{3mm}nU^{n}_{i}\hspace{3mm}} \cfk_{n}(\cH)\  \]

\noindent
Thus, $U_{i}^{n}=0$ on homology.

\end{proof}

Also like $\sln$ homology, disjoint union corresponds to tensor product:

\begin{lem}

Let $(Y_{1}, L_{1})$, $(Y_{2}, L_{2})$ be two links. Then 
\[ \cfk_{n}(Y_{1} \# Y_{2}, L_{1} \sqcup L_{2}) \cong \cfk_{n}(Y_{1}, L_{1}) \otimes \cfk_{n}(Y_{2}, L_{2}) \]

\noindent
where $(Y_{1} \# Y_{2}, L_{1} \sqcup L_{2})$ is the link obtained by taking the connected sum of $(Y_{1},L_{1})$ with $(Y_{2},L_{2})$ away from the links $L_{1}$ and $L_{2}$.

\end{lem}

\begin{proof}

Let $\cH_{1}$ and $\cH_{2}$ be punctured Heegaard diagrams for $L_{1}$ and $L_{2}$. Define $\cH_{3}$ to be the punctured Heegaard diagram for $L_{1} \sqcup L_{2}$ obtained by taking the connected sum of $\cH_{1}$ and $\cH_{2}$ at the punctures in the respective diagrams. The puncture $p$ in $\cH_{3}$ is placed on the neck of the connected sum.

This puncture blocks any interaction between the two components of the connected sum, and the resulting complex is given by 
\[ \cfk_{n}(\cH_{3}) \cong \cfk_{n}(\cH_{1}) \otimes \cfk_{n}(\cH_{2}) \]

\noindent
The lemma follows.

\end{proof}

The final property that we will prove in this section is the analog of Rasmussen's spectral sequences from HOMFLY-PT homology to $\sln$ homology.

\begin{lem}

For all $n\ge 1$ and any link $L$, there is a spectral sequence from $\hfk(L)$ to $\hfk_{n}(L)$.

\end{lem}

\begin{proof}

Let $\gr_{\alg}$ be the grading $M-2A$. The differential $\partial_{n}$ is not homogeneous with respect to this differential - a homotopy class $\phi$ increases $\gr_{\alg}$ by $1-n_{\bf{z}}(\phi)$. Thus, this grading induces a spectral sequence $(E_{k},d_{k})$ on $\cfk_{n}(L)$ where the differential $d_{k}$ is the piece of the total differential which increases $\gr_{\alg}$ by $1-k$. Since multiplication by the $U_{i}$ preserves $\gr_{\alg}$, the complex is bounded with respect to this grading, so the spectral sequence converges to the total homology $\hfk_{n}(L)$. The differential $d_{0}$ counts holomorphic discs with multiplicity zero at the $\bf{z}$ basepoints, so $(E_{0},d_{0}) = \cfk(L)$, and $E_{1}=\hfk(L)$.

\end{proof}

\section{Generalization to Singular Links}\label{section3} In \cite{OSS}, Ozsv\'{a}th, Stipsicz, and Szab\'{o} define a Floer homology theory for singular links in $S^{3}$. Ozsv\'{a}th and Szab\'{o} modify this construction in \cite{Szabo} to give an oriented cube of resolutions for knot Floer homology. 

\begin{rem}

Ozsv\'{a}th and Szab\'{o} use twisted coefficients in their construction to get regular sequences, but the same construction is defined by Manolescu in \cite{Manolescu} with untwisted coefficients, which is the setting we will be working in.

\end{rem}

The Heegaard diagram for a singular link is obtained by relaxing the conditions on a normal Heegaard diagram. 

\begin{defn}

A \emph{singular Heegaard diagram} is a tuple $(\Sigma, \alpha, \beta, \bf{w}, \bf{z})$ such that $(\Sigma, \alpha, \beta)$ is a Heegaard splitting for $S^{3}$, and each component of $\Sigma  \backslash \alpha$ contains $k$ $\bf{w}$-basepoints and $k$ $\bf{z}$-basepoints, with $k \in \{1,2\}$. Similarly for $\Sigma  \backslash \beta$. Additionally, if $\Sigma  \backslash \alpha$ contains two of each basepoint, then we require the two $\bf{z}$ basepoints to lie in the same region of $\Sigma \backslash (\alpha \cup \beta)$. 

\end{defn}

Such a pair of $\bf{z}$-basepoints is called an \emph{inseparable pair}. The tuple $(\Sigma, \alpha, \beta, \bf{w}, \bf{z})$ determines a singular link in $S^{3}$, with the inseparable pairs corresponding to singular points. If $z_{i}$ and $z_{j}$ are an inseparable pair, we can replace them with a single basepoint $z_{i,j}$ - the two subscripts indicate that it is a double point, so it is connected to two $\bf{w}$ basepoints on each side. Let $w_{a_{1}(i,j)}$ and $w_{a_{2}(i,j)}$ be the two basepoints on the $\alpha$ side, and $w_{b_{1}(i,j)}$ and $w_{b_{2}(i,j)}$ the two basepoints on the $\beta$ side.

Let $\cH$ be a singular Heegaard diagram for a singular link $S$. One can define a chain complex $\cfk^{-}(\cH)$ over in the usual way, defining the differential to be 
\[  \partial^{-}(x) = \sum_{y \in \mathbb{T_{\alpha}} \cap \mathbb{T_{\beta}}} \sum_{\substack{\phi \in \pi_{2}(x,y) \\ \mu(\phi)=1 \\ n_{\bf{z}^{(1)}}(\phi)=0 \\   n_{\bf{z}^{(2)}}(\phi)=0} }   \# \widehat{\mathcal{M}}(\phi) U_{1}^{n_{w_{1}}(\phi)}\cdot \cdot \cdot U_{k}^{n_{w_{k}}(\phi)} y \]

\noindent
where the $\bf{z}^{(1)}$ refers to the normal $z_{i}$ basepoints and $\bf{z}^{(2)}$ refers to the double points. The relative Maslov and Alexander gradings on this complex are given by 
\[ M(x) - M(y) = \mu(\phi) +2n_{\bf{z}^{(2)}}(\phi) - 2n_{\bf{w}}(\phi) \]
\[ A(x) - A(y) = n_{\bf{z}^{(1)}}(\phi)+2n_{\bf{z}^{(2)}}(\phi) - n_{\bf{w}}(\phi) \]

\noindent
The definition of $\cfk^{-}(\cH)$ extends in the obvious way to punctured singular Heegaard diagrams using the same rules as in the non-singular case. Define $\cfk(S)$ to be the chain complex $\cfk^{-}(\cH)$ where $\cH$ is a punctured singular Heegaard diagram for $S$.

The modification to this complex made in \cite{Szabo} is the addition of a Koszul complex. For a singular point $z_{ij}$, let $L(i,j)$ be the linear element $U_{a_{1}(i,j)}+U_{a_{2}(i,j)}-U_{b_{1}(i,j)}-U_{b_{2}(i,j)}$. The Koszul complex is given by

\[ K(S) = \bigotimes_{z_{ij} \in \bf{z}^{(2)}} R \xrightarrow{ \hspace{5mm} L(i,j) \hspace{5mm}} R \]

\noindent
where $R=\Q[U_{1},...,U_{k}]$.

\begin{defn} The total complex $C_{F}(S)$ is the tensor product $\cfk(S) \otimes_{R} K(S) $. Note that when $S$ has no singular points, $C_{F}(S) = \cfk(S)$.

\end{defn}

\begin{thm}[\hspace{1sp}\cite{Manolescu},\cite{Szabo}] \label{exacttriangle}

Let $D_{+}$, $D_{x}$, and $D_{s}$ be three diagrams which differ at a single crossing $c$, with $D_{+}$ having a positive crossing at $c$, $D_{x}$ having a singular point at $c$, and $D_{s}$ having the oriented smoothing at $c$. Then there are exact triangles 

\begin{figure}[H]
\centering
\begin{tikzpicture}
\node(x){$C_{F}(D_{x})$};
\node(x') at ([shift={(14,0)}]x){$C_{F}(D_{x})$};
\node (s) at ([shift={(5,0)}]x) {$C_{F}(D_{s})$};
\node (s') at ([shift={(9,0)}]x) {$C_{F}(D_{s})$};
\node (+) at ([shift={(2.7,-2.3)}]x) {$C_{F}(D_{+})$};
\node (-) at ([shift={(11.7,-2.3)}]x) {$C_{F}(D_{-})$};
\path[-stealth] 
  (x) edge (s)
  (s') edge (x')
  (x') edge (-)
  (-) edge (s')
  (s) edge (+)
  (+) edge (x);
\end{tikzpicture}
\end{figure}

\end{thm}

\noindent
Iterating this exact triangle gives an oriented cube of resolutions for any link $L$.

We will now go about constructing our complex for singular links in $S^{3}$, which will be obtained by adding differentials to both $\cfk(S)$ and $K(S)$. Let $Q(i,j)$ be the quadratic element $U_{a_{1}(i,j)}U_{a_{2}(i,j)}-U_{b_{1}(i,j)}U_{b_{2}(i,j)}$. In \cite{KR}, Khovanov and Rozansky define polynomials $p_{1}(i,j,n), p_{2}(i,j,n)$ in $R$ of degree $n-1$ and $n-2$, respectively, such that 
\[ L(i,j)p_{1}(i,j,n) + Q(i,j)p_{2}(i,j,n) = U^{n}_{a_{1}(i,j)}+U^{n}_{a_{2}(i,j)}-U^{n}_{b_{1}(i,j)}-U^{n}_{b_{2}(i,j)} \]

\noindent
We define the ground ring $\Q[\{U_{i}\} \cup \{V_{i} \} \cup \{W_{ij}\}]$ where the $U_{i}$ correspond to the $w_{i}$, the $V_{i}$ correspond to the $z_{i}$, and the $W_{ij}$ correspond to the double points $z_{ij}$. The master complex $\cfk_{U,V,W}(S)$ is defined over this ring with differential $\partial_{U,V,W}$, given by

\[  \partial_{U,V,W}(x) = \sum_{y \in \mathbb{T_{\alpha}} \cap \mathbb{T_{\beta}}} \sum_{\substack{\phi \in \pi_{2}(x,y) \\ \mu(\phi)=1 \\ n_{p}(\phi)=0}}    \# \widehat{\mathcal{M}}(\phi) U^{n_{\bf{w}}(\phi)}V^{n_{\bf{z(1)}}(\phi)} W^{n_{\bf{z}(2)}(\phi)} y \]

To obtain the complex $\cfk_{n}(S)$, we quotient this complex by the following relations in the ground ring:
\[ V_{i} = (U_{a(i)}^{n} - U_{b(i)}^{n})/(U_{a(i)}-U_{b(i)}) \]
\[ W_{ij} = p_{2}(i,j,n) \]

\noindent
Note that the resulting ground ring is isomorphic to the ring $R=\Q[U_{1},...,U_{k}]$. We will also add differentials to the Koszul complex $K(S)$ to make it a matrix factorization $K_{n}(S)$:

\[K_{n}(S) = \bigotimes_{z_{ij} \in \bf{z}^{(2)}} R \mathrel{ \substack{\xrightarrow{\hspace{5mm}L(i,j) \hspace{5mm}} \\[-.4ex] \xleftarrow[\hspace{3.5mm}p_{1}(i,j,n) \hspace{3.5mm}]{}}} R \] 

\begin{defn}

The complex $C_{F(n)}(S)$ is defined to be $\cfk_{n}(S) \otimes K_{n}(S)$. The differential on this complex is denoted $\partial_{n}$.

\end{defn}

\begin{lem}

The differential $\partial_{n}$ satisfies $\partial_{n}^{2}=0$.

\end{lem}

\begin{proof}

We will start by computing $\partial_{n}^{2}$ on the $\cfk_{n}(S)$ component. Each $\alpha$- or $\beta$-degeneration contains a unique $z_{i}$ or $z_{ij}$ basepoint or the puncture $p$, so it suffices to sum the degenerations over the $z_{i}$ and $z_{ij}$.

For $z_{i} \in \bf{z}^{(1)}$, the $\alpha$-degeneration containing $z_{i}$ has contribution $U_{a(i)}V_{i}$ and the $\beta$-degeneration has contribution $-U_{b(i)}V_{i}$, so the total contribution to $\partial_{n}^{2}$ is $(U_{a(i)}-U_{b(i)})V_{i}=U_{a(i)}^{n} - U_{b(i)}^{n}$.

For $z_{ij} \in \bf{z}^{(2)}$, the $\alpha$-degeneration containing $z_{ij}$ has contribution $U_{a_{1}(i,j)}U_{a_{2}(i,j)}W_{ij}$ and the $\beta$-degeneration has contribution $-U_{b_{1}(i,j)}U_{b_{2}(i,j)}W_{ij}$, so the total contribution to $\partial_{n}^{2}$ is $Q(i,j)W_{ij}=Q(i,j)p_{2}(i,j,n)$.

On the $K_{n}(S)$ side, each $z_{ij}$ in $\bf{z}^{(2)}$ contributes $L(i,j)p_{1}(i,j,n)$ to $\partial_{n}^{2}$. 

Thus, for each $z_{ij} \in \bf{z}^{(2)}$, the total contribution to $\partial_{n}^{2}$ in $\cfk_{n}(S) \otimes K_{n}(S)$ is $L(i,j)p_{1}(i,j,n) + Q(i,j)p_{2}(i,j,n)$, or $U^{n}_{a_{1}(i,j)}+U^{n}_{a_{2}(i,j)}-U^{n}_{b_{1}(i,j)}-U^{n}_{b_{2}(i,j)}$. Adding up the contributions from both the $z_{i}$ and $z_{ij}$ basepoints, we see that each $U_{i}^{n}$ appears twice, once with positive sign and once with negative sign. The lemma follows.

\end{proof}

Thus, $C_{F(n)}(S)$ is a well-defined chain complex - we denote its homology by $H_{F(n)}(S)$. As with non-singular links, the complex is not homogeneous with respect to the Maslov and Alexander grading, but it is homogeneous of degree $n$ with respect to the grading $\gr_{n}=nM-2(n-1)A$.

\begin{defn}

A diagram for a singular knot is \emph{completely singular} if it contains no crossings.

\end{defn}

For completely singular links, $H_{F(n)}(S)$ agrees with our definition in \cite{Me2}, where we study a homology theory defined only for completely singular links.

\begin{lem}[\hspace{1sp}\cite{Me2}] \label{isolemma}

For a completely singular link $S$, the homology $H_{F(n)}(S)$ is isomorphic to the $\mathfrak{sl}_{n}$ homology $H_{n}(S)$ as graded vector spaces.

\end{lem}

\begin{conj} \label{conj3.9}

Let $D_{+}$, $D_{x}$, and $D_{s}$ be three diagrams which differ at a single crossing $c$, with $D_{+}$ having a positive crossing at $c$, $D_{x}$ having a singular point at $c$, and $D_{s}$ having the oriented smoothing at $c$. Then there is an exact triangle 

\begin{figure}[H]
\centering
\begin{tikzpicture}
\node(x){$C_{F(n)}(D_{x})$};
\node(x') at ([shift={(14,0)}]x){$C_{F(n)}(D_{x})$};
\node (s) at ([shift={(5,0)}]x) {$C_{F(n)}(D_{s})$};
\node (+) at ([shift={(2.7,-2.3)}]x) {$C_{F(n)}(D_{+})$};
\node (s') at ([shift={(9,0)}]x) {$C_{F(n)}(D_{s})$};
\node (-) at ([shift={(11.7,-2.3)}]x) {$C_{F(n)}(D_{-})$};
\path[-stealth] 
  (x) edge (s)
  (s) edge (+)
  (s') edge (x')
  (x') edge (-)
  (-) edge (s')
  (+) edge (x);
\end{tikzpicture}
\end{figure}

\noindent
and if $D_{x}, D_{s}$ are completely singular, then the map between them on homology is the $\mathfrak{sl}_{n}$ edge map. 

\end{conj}

The existence of the conjectured spectral sequences from $H_{n}(K)$ to $\hfk_{n}(K)$ would follow immediately from this conjecture. Unfortunately, the fact that the $\bf{z}$ basepoints play different roles in the three diagrams prevents the same argument from Theorem \ref{exacttriangle} from working.

\newpage

\section{Khovanov-Rozansky Homology} In this section we will give a short background on the properties of Khovanov-Rozansky homology. A definition of the complexes can be found in \cite{KR}, \cite{KR2}, and \cite{Rasmussen}.

\subsection{Knot Polynomials}

 The Khovanov-Rozansky package includes HOMFLY-PT homology and $\mathfrak{sl}_{n}$ homology, which are categorifications of the HOMFLY-PT polynomial and $\mathfrak{sl}_{n}$ polynomial, respectively. The HOMFLY-PT polynomial  of a link $L$ in $S^{3}$ is a two variable polynomial $P(a,q)(L)$, and it is determined by the following skein relation:
\[ aP(a,q)(D_{+}) -a^{-1}P(a,q)(D_{-})= (q-q^{-1})P(a,q)(D_{s}) \]

\noindent
together with the normalization $P(a,q, \unknot)=\frac{a-a^{-1}}{q-q^{-1}}$ (\hspace{1sp}\cite{HOMFLY}, \cite{PT}). The polynomial $P_{n}$ is given by the specialization $a=q^{n}$:
\[ P_{n}(q,L) = P(q^{n},q,L) \]

\noindent
There are also reduced polynomials $\overline{P}(a,q,L)$ and $\overline{P}_{n}(q,L)$ obtained from the same skein relation, but with the normalization $\overline{P}(a,q,\unknot) = \overline{P}_{n}(q,\unknot)=1$. For $n \ge 1$, $P_{n}(q,L)$ and $\overline{P}_{n}(q,L)$ are the unreduced and reduced $\sln$ polynomials, respectively. $\overline{P}_{0}(q,L)$ is the Alexander polynomial of $L$.

%It turns out that for $n= 1$, these polynomials are particularly simple:

%\begin{lem}
%
%For any link $L$ with $l$ components, 
%\[P_{1}(q,L)=\overline{P}_{1}(q,L)=1\]
%
%\end{lem}

\subsection{HOMFLY-PT Homology} Let $L$ be a link in $S^{3}$, and let $D$ be a connected, decorated braid diagram for $L$. The diagram $D$ can be viewed as an oriented graph in which the crossings are 4-valent vertices and the marked point (the decoration) is a bivalent vertex. Let $e_{1},...,e_{k}$ be the edges of $D$, and let $R_{0}=\Q[U_{1},...,U_{k}]$. Each 4-valent vertex $v$ in $D$ has two incoming edges, which we will denote $e_{b_{1}(v)}$ and $e_{b_{2}(v)}$, and two outgoing edges $e_{a_{1}(v)}$ and $e_{a_{2}(v)}$. Define the linear and quadratic terms $L(v)$ and $Q(v)$ by 
\[L(v)=U_{a_{1}(v)}+U_{a_{2}(v)}-U_{b_{1}(v)}-U_{b_{2}(v)} \hspace{10mm} Q(v)= U_{a_{1}(v)}U_{a_{2}(v)}-U_{b_{1}(v)}U_{b_{2}(v)} \]

\begin{defn}

The ground ring $R$ is defined to be the quotient \[R_{0}/\{L(v)=0 \text{ for } v \in V_{4}(D) \} \]

\noindent
where $V_{4}(D)$ is the set of 4-valent vertices in $D$. 

\end{defn}

The HOMFLY-PT complex is defined by writing down a complex for each crossing in $D$, then taking the tensor product over all of these crossing complexes. The complex is triply graded, with gradings $(\mathbf{gr}_{q}, \mathbf{gr}_{h},\mathbf{gr}_{v})$. We use the bold fonts to distinguish the Khovanov-Rozansky gradings from the knot Floer gradings. Each $U_{i}$ has grading $(2,0,0)$.

Let $v$ be a positive crossing in $D$. The complex $C(v)$ is given by  

\begin{figure}[!h]
\centering
\begin{tikzpicture}
  \matrix (m) [matrix of math nodes,row sep=5.5em,column sep=6.5em,minimum width=2em] {
     R\{2,-2, -2\} & R\{0,0,-2\} \\
     R\{0, -2, 0\} & R\{0,0,0\} \\};
  \path[-stealth]
    (m-1-1) edge node [left] {$U_{a_{1}(v)}-U_{b_{2}(v)}$} (m-2-1)
    (m-1-1) edge node [above] {$Q(v)$} (m-1-2)
    (m-1-2) edge node [right] {$1\hspace{25mm}$} (m-2-2)
    (m-2-1) edge node [above] {$U_{a_{1}(v)}-U_{b_{1}(v)}$} (m-2-2);
\end{tikzpicture}
\end{figure}

\newpage
If $v$ is a negative crossing, the complex $C(v)$ is given by 

\begin{figure}[!h]
\centering
\begin{tikzpicture}
  \matrix (m) [matrix of math nodes,row sep=5.5em,column sep=6.5em,minimum width=2em] {     
     R\{0, -2, 0\} & R\{0,0,0\} \\
     R\{0,-2, 2\} & R\{-2,0, 2\} \\};
  \path[-stealth]
    (m-1-1) edge node [left] {$\hspace{20mm}1$} (m-2-1)
    (m-1-1) edge node [above] {$U_{a_{1}(v)}-U_{b_{1}(v)}$} (m-1-2)
    (m-1-2) edge node [right] {$U_{a_{1}(v)}-U_{b_{2}(v)}$} (m-2-2)
    (m-2-1) edge node [above] {$Q(v)$} (m-2-2);
\end{tikzpicture}
\end{figure}

Both of these complex come with a vertical filtration - let $d_{+}$ denote the differentials which preserve this filtration and $d_{v}$ the differentials which decrease it. This filtration corresponds to the oriented cube of resolutions, where 
\[ R \xrightarrow{ \hspace{8mm} Q(v) \hspace{8mm} }R \]

\noindent
is the complex for the singularization, and 
\[ R \xrightarrow{ \hspace{2mm} U_{a_{1}(v)}-U_{b_{1}(v)}\hspace{2mm} }R \]

\noindent
is the complex for the oriented smoothing. For this reason, we will also refer to the vertical filtration as the \emph{cube} filtration. Viewing these complexes from the cube of resolutions perspective, $d_{+}$ is the vertex maps, while $d_{v}$ is the edge maps. Note that $d_{+}$ is homogeneous of degree $(2, 2, 0)$ and $d_{v}$ is homogeneous of degree $(0, 0, 2)$.

\begin{defn}

The \emph{middle} HOMFLY-PT complex is given by 

\[ C^{M}(D) = \Big[ \bigotimes_{v \in V_{4}(D)}  C(v) \Big]  \{-w + b, w + b - 1, w - b + 1\} \]

\noindent
where $b$ and $w$ are the braid index and the writhe of $D$, respectively. The middle HOMFLY-PT homology is $ H^{M}(L) = H_{*}(H_{*}(C^{M}(D), d_{+}), d_{v}^{*}) $.

\end{defn}

We will be more interested in two variants of this, the unreduced and reduced theories.

\begin{defn} The \emph{reduced} HOMFLY-PT complex is given by
\[ \overline{C}(D) = C^{M}(D) \otimes \big[ R\{1, 0,-2\} \xrightarrow{ \hspace{4mm} U_{1} \hspace{3mm}} R\{-1,0,0\} \big] \]

\noindent
The reduced HOMFLY-PT homology is $\overline{H}(L) = H_{*}(H_{*}(\overline{C}(D), d_{+}), d_{v}^{*}) $. Note that the $U_{1}$ differential above is part of $d_{v}$.

\end{defn}

\begin{defn} The \emph{unreduced} HOMFLY-PT complex is given by
\[ C(D) = C^{M}(D) \otimes \big[ R\{0,-1,-1\} \oplus R\{0,1,-1\} \big] \]

\noindent
The unreduced HOMFLY-PT homology is $H(L) = H_{*}(H_{*}(C(D), d_{+}), d_{v}^{*}) $.
\end{defn}

\begin{lem}
The graded Euler characteristic of reduced HOMFLY-PT homology is the reduced HOMFLY-PT polynomial, and similarly for the unreduced versions.
\[ \sum_{i,j,k} (-1)^{(k-j)/2} a^{j} q^{i} \dim(\overline{H}^{i,j,k}(L)) = \overline{P}(a,q,L) \]

\[ \sum_{i,j,k} (-1)^{(k-j)/2} a^{j} q^{i} \dim(H^{i,j,k}(L)) = P(a,q,L) \]

\end{lem}

All three of these homology theories can be viewed as the $E_{2}$ page of the spectral sequence induced by the cube filtration. The $E_{\infty}$ page is the homology with respect to the total differential $d_{+}+d_{v}$, which Rasmussen calls the ``$\mathfrak{sl}_{-1}$" homology because it categorifies $H_{-1}(L)$.

%
%\begin{defn}
%
%
%The reduced and unreduced $\mathfrak{sl}_{-1}$ homologies of a link $L$ are defined to be
%\[ \overline{H}_{-1}(L) = H_{*}(\overline{C}_{H}(D),d_{+}+d_{v}) \hspace{3mm} \text{ and } \hspace{3mm} H_{-1}(L) = H_{*}(C_{H}(D),d_{+}+d_{v})\]
%
%\end{defn}
%
%\noindent
%These homology theories are bigraded, with bigrading $(\gr_{q}-\gr_{h}, \gr_{h}+gr_{v})$.
%
%\begin{lem}[\hspace{1sp}\cite{Rasmussen}]
%
%If $L$ is an $l$ component link in $S^{3}$, then the reduced (resp. unreduced) $\mathfrak{sl}_{-1}$ homology of $L$ is the reduced (resp. unreduced) HOMFLY-PT homology of the $l$-component unlink. 
%
%\end{lem}
%
%\noindent
%It follows that the graded Euler characteristic $\mathfrak{sl}_{1}$ homology is the $\mathfrak{sl}_{1}$ polynomial for both the reduced and unreduced versions.
%

\subsection{$\sln$ Homology} The $\sln$ chain complexes are obtained by adding additional differentials to the HOMFLY-PT complex. If $v$ is a negative crossing in $D$, define $C_{n}(v)$ by 

\begin{figure}[H]
\centering
\begin{tikzpicture}
  \matrix (m) [matrix of math nodes,row sep=6.5em,column sep=8em,minimum width=2em] {
     R\{2,-2, -2\} & R\{0,0,-2\} \\
     R\{0, -2, 0\} & R\{0,0,0\} \\};
  \path[-stealth]
    (m-1-1) edge node [left] {$U_{a_{1}(v)}-U_{b_{2}(v)}$} (m-2-1)
            edge [bend left = 15] node [above] {$Q(v)$} (m-1-2)
    (m-2-1) edge [bend left = 15] node [above] {$U_{a_{1}(v)}-U_{b_{1}(v)}$} (m-2-2)
    (m-1-2) edge node [right] {$1\hspace{27mm}$} (m-2-2)
            edge [bend left=15] node [below] {$p_{2}(v,n+1)$} (m-1-1)
    (m-2-2) edge [bend left=15] node [below]  {$p_{1}(v,n+1)$} (m-2-1);        
\end{tikzpicture}
\end{figure}

\noindent
where $p_{2}(v,n+1)$ and $p_{1}(v,n+1)$ are chosen such that 
\[ Q(v)p_{2}(v,n+1) = (U_{a_{1}(v)}-U_{b_{1}(v)})p_{1}(v) = U^{n+1}_{a_{1}(v)}+U^{n+1}_{a_{2}(v)}-U^{n+1}_{b_{1}(v)}-U^{n+1}_{b_{2}(v)} \in R \]

\noindent
Note that these choices are possible because $R$ is a quotient of $R_{0}$ by the ideal on the linear elements - it would not be possible in $R_{0}$.

Similarly, if $v$ is a negative crossing in $D$, define $C_{n}(v)$ by

\begin{figure}[!h]
\centering
\begin{tikzpicture}
  \matrix (m) [matrix of math nodes,row sep=6.5em,column sep=8em,minimum width=2em] {     
     R\{0, -2, 0\} & R\{0,0,0\} \\
     R\{0,-2, 2\} & R\{-2,0, 2\} \\};
  \path[-stealth]
    (m-1-1) edge node [left] {$\hspace{20mm}1$} (m-2-1)
            edge [bend left = 15] node [above] {$U_{a_{1}(v)}-U_{b_{1}(v)}$} (m-1-2)
    (m-2-1) edge [bend left = 15] node [above] {$Q(v)$} (m-2-2)
    (m-1-2) edge node [right] {$U_{a_{1}(v)}-U_{b_{2}(v)}$} (m-2-2)
            edge [bend left=15] node [below] {$p_{1}(v,n+1)$} (m-1-1)
    (m-2-2) edge [bend left=15] node [below]  {$p_{2}(v,n+1)$} (m-2-1);        
\end{tikzpicture}
\end{figure}

As with HOMFLY-PT homology, define $C_{n}^{M}(D)$ to be the tensor product of these complexes over all the crossings in $D$:

\[ C^{M}_{n}(D) = \Big[ \bigotimes_{v \in V_{4}(D)}  C_{n}(v) \Big]  \{-w + b, w + b - 1, w - b + 1\} \]

\noindent
The new differentials (with coefficient $p_{1}$ or $p_{2}$) have triple grading $(2n - 2, -2, 0)$. We will call these the $d_{-}$ differentials, so that the total differential is given by $d_{+}+d_{-}+d_{v}$.

\begin{lem}

The total differential $d_{+}+d_{-}+d_{v}$ satisfies $(d_{+}+d_{-}+d_{v})^{2}=0$.

\end{lem}

\begin{proof} Each complex $C_{n}(v)$ is a matrix factorization with potential $w_{n}(v)= U^{n+1}_{a_{1}(v)}+U^{n+1}_{a_{2}(v)}-U^{n+1}_{b_{1}(v)}-U^{n+1}_{b_{2}(v)}$. Taking the tensor product of matrix factorizations corresponds to summing potentials, and the sum of these potentials is zero in $R$.
\end{proof}

 Define $\mathbf{gr}_{n}$ to be the grading 
\[ \mathbf{gr}_{n} = \mathbf{gr}_{q} + \frac{n-1}{2} \mathbf{gr}_{h} \]

\noindent
The sum $d_{+}+d_{-}$ is homogeneous of degree $n+1$ with respect to $\mathbf{gr}_{n}$.

\begin{defn}
The \emph{reduced} $\sln$ complex is given by 
\[ \overline{C}_{n}(D) = C^{M}_{n}(D) \otimes \big[ R\{1, -2\} \xrightarrow{ \hspace{4mm} U_{1} \hspace{3mm}} R\{-1,0\} \big] \]

\noindent
where the bigrading is given by $(\mathbf{gr}_{n}, \mathbf{gr}_{v})$. The reduced $\sln$ homology is 
\[ \overline{H}_{n}(L)=H_{*}(H_{*}(\overline{C}_{n}(D), d_{+}+d_{-}),d_{v}^{*})\]
\end{defn}

Let $v_{0}$ denote the marked bivalent vertex in $D$, and let $e_{a(v_{0})},e_{b(v_{0})}$ be the outgoing and incoming edges at $v_{0}$, respectively. 

\begin{defn}

The \emph{unreduced} $\sln$ complex is given by 
\[ C_{n}(D) = C^{M}_{n}(D) \otimes \big[ R\{\frac{-n+1}{2}, -1\} \xleftarrow{ \hspace{4mm} (n+1)U^{n}_{a(v_{0})} \hspace{3mm}} R\{\frac{n-1}{2},-1\} \big] \]

\end{defn}

\noindent
Note that in $R$, $U_{a(v_{0})}=U_{b(v_{0})}$, so we could just has easily chosen the incoming edge instead of the outgoing edge.

\begin{lem}
The graded Euler characteristic of reduced $\sln$ homology is the reduced $\sln$ polynomial, and similarly for the unreduced versions.
\[ \sum_{I,J} (-1)^{J/2} q^{I} \dim(\overline{H}_{n}^{I,J}(L)) = \overline{P}_{n}(q,L) \]
\[ \sum_{I,J} (-1)^{J/2} q^{I} \dim(H_{n}^{I,J}(L)) = P_{n}(q,L) \]

\end{lem}

%
%
%As with HOMFLY-PT homology, the $\sln$ homology theories can be viewed as the $E_{2}$ page of the spectral sequence induced by the cube filtration. These spectral sequences converge to truncated versions of $\mathfrak{sl}_{-1}$ homology. 
%
%\begin{lem} Let $D$ be a diagram for an $l$ component link in $S^{3}$. Then
%
%\[ H_{*}(\overline{C}_{n}(D), d_{+}+d_{-}+d_{v}) \cong \big( \Q[U]/U^{n}=0  \big) ^{\otimes l-1}\]
%\[ H_{*}(C_{n}(D), d_{+}+d_{-}+d_{v}) \cong \big( \Q[U]/U^{n}=0  \big) ^{\otimes l}\]
%
%\end{lem}
%

\subsection{Properties of Khovanov-Rozansky Homology}

The HOMFLY-PT and $\sln$ homologies satisfy many structural properties analogous to $\cfk$ and $\cfk_{n}$. For any $l$-component link $L$ in $S^{3}$, we have:

$\bullet$ $H(L)$ is a finitely generated, free $\Q[U_{1},...,U_{l}]$-module.

$\bullet$ $H_{n}(L)$ is a finitely generated module (not necessarily free) over 

\[\Q[U_{1},...,U_{l}]/(U_{1}^{n}=...=U_{l}^{n}=0)\]

$\bullet$ For all $n$, there is a spectral sequence from $H(L)$ to $H_{n}(L)$.

\vspace{1mm}

\noindent
The reduced complexes satisfy the same properties, with some stronger properties for knots:
\vspace{-1mm}

$\bullet$ $\overline{H}(K)$ is finite-dimensional over $\Q$.

$\bullet$ There is an integer $n(K)$ such that for all $n \ge n(K)$, the spectral sequence from $\overline{H}(K)$ to $\overline{H}_{n}(K)$ is trivial, i.e. $\dim(\overline{H}(K)) = \dim(\overline{H}_{n}(K))$.

\vspace{1mm}

\section{The Conjectured Spectral Sequences} Inspired by Lemma \ref{isolemma}, we conjecture the following relationship between $\sln$ homology and $\hfk_{n}$.

\begin{conj} \label{conj5.1}

For all $n \ge 1$, and all $L \in S^{3}$ there are spectral sequences from $H_{n}(L)$ to $\hfk_{n}(L)$ and from $\overline{H}(L)$ to $\widehat{\hfk}_{n}(L)$. The grading $\mathbf{gr}_{n}+\frac{n}{2}\mathbf{gr}_{v}$ on $H_{n}(L)$ descends to the grading $\gr_{n}$ on $\hfk_{n}(L)$ (and similarly for the reduced versions).

\end{conj}

\noindent
In the case $n=2$, this is the $\delta$-grading on Khovanov homology, which agrees with the conjecture from \cite{Rasmussen3}. Note that Conjecture \ref{conj5.1} follows from Conjecture \ref{conj3.9}.

\subsection{Computational Evidence} In this section we will describe classes of knots for which the conjecture is true.

\subsubsection{n=1:}

\begin{lem}

For the $\mathfrak{sl}_{1}$ case, $H_{1}(L)=\hfk_{1}(L)=\Q\{0\}$ for any link $L$. Similarly, $\overline{H}_{1}(L)=\widehat{\hfk}_{1}(L)=\Q\{0\}$.

\end{lem}

\subsubsection{n=2:}

\begin{lem} [\hspace{1sp}\cite{Rasmussen3}]

For any 2-bridge knot $K$, \[\overline{H}_{2}(K) \cong \widehat{\hfk}_{2}(K)\] as $\delta$-graded vector spaces.

\end{lem}

We will prove that this is the case for the unreduced theories as well:

\begin{lem}

For any 2-bridge knot $K$, \[H_{2}(K) \cong \hfk_{2}(K)\] as $\delta$-graded $\Q[U]$-modules.

\end{lem}

\begin{proof} Let $r$ denote the rank of $\overline{H}_{2}(K)$. We can write $\overline{H}_{2}(L)=\Q^{r}\{\sigma(K)\}$, where $\sigma(K)$ is the signature of $K$. In \cite{MeandAkram}, Alishahi and the author show that 
\[H_{2}(K) = \Q^{(r+1)/2} \{\sigma(K)-1\} \oplus \Q^{(r+1)/2} \{\sigma(K)+1\} \]

\noindent
Recording the module structure as well, this can be written, 
\[ H_{2}(K) = [2]\{\sigma(K)+1 \} \oplus  \frac{r-1}{2}[1]\{\sigma(K)-1\} \oplus \frac{r-1}{2}[1] \{\sigma(K)+1\}  \]

We must now compute the corresponding homology for $\hfk_{2}(L)$. The reduced homology $\widehat{\hfk}_{2}(K)$ is given by $\Q^{r} \{\sigma(K)\}$. If we were to allow discs to pass through the $\bf{w}$ and $\bf{z}$ basepoints, this would induce the $\tau$ spectral sequence on $K$, which converges to $\Q\{\sigma(K)\}$. Since the generators lie in a single $\delta$-grading and these new discs must have Maslov index 1, the new discs must have multiplicity 1 at $\bf{w}$ and 0 at $\bf{z}$ or vice versa.

The unreduced homology $\hfk_{2}(K)$ can be obtained from $\widehat{\hfk}_{2}(K)$ by tensoring with $\Q[U]/(U^{2}=0)\{-1\}$, and adding additional differentials corresponding to the discs which pass through $\bf{w}$ or $\bf{z}$. The $\tau$ spectral sequence tells us that we can find a basis $\mathcal{B}=\{x_{1},x_{1},...,x_{r-1},y \}$ for $\widehat{\hfk}_{2}(K)$ such that this differential is given by

\[ \partial(x_{i}) = \begin{cases} 
     Ux_{i+1} & \textrm{ if $i$ is odd} \\
      0 &  \textrm{ if $i$ is even} \\ 
   \end{cases} \]
   
\[ \partial(y) = 0 \hspace{34mm}\]

\noindent
Thus, the generators of homology in grading $\sigma(K)+1$ are $y$ and $x_{i}$ for $i$ odd, and the generators of homology in grading $\sigma(K)-1$ are $Uy$ and $Ux_{i}$ for $i$ even, giving
\[\hfk_{2}(K) = \Q^{(r+1)/2} \{\sigma(K)-1\} \oplus \Q^{(r+1)/2} \{\sigma(K)+1\} \]

\noindent
Since multiplication by $U$ is trivial on all generators of homology except $y$, the module structure is given by 
\[\hfk_{2}(K) = [2]\{\sigma(K)+1 \} \oplus  \frac{r-1}{2}[1]\{\sigma(K)-1\} \oplus \frac{r-1}{2}[1] \{\sigma(K)+1\}  \]

\noindent
where the $[2]$ is generated by $y$. This proves the lemma.

\end{proof}

\begin{rem}

Techinically, the $\tau$ spectral sequence is induced by allowing discs to pass through just the $z$ basepoint rather than $z$ and $w$. However, there is a spectral sequence from the former to the latter, and since the former is already just $\Q$, the two homologies must be isomorphic.

\end{rem}

\begin{lem}

If $K$ is the $(3,m)$ torus knot, then there is a spectral sequence from $H_{2}(K)$ to $\hfk_{2}(K)$.

\end{lem}

\begin{proof}

Assume $m=3k+1$ for some $k$. From \cite{Turner}, the Poincare polynomial of $\delta$-graded $H_{2}(T_{3,3k+1})$ is given by
\begin{equation} \label{eq5.4}
 q^{4k-1}(1+4q^{2} + 6q^{4}+...+6q^{2k-2}+6q^{2k}+3q^{2k+2}) 
\end{equation}

\noindent
where the $\delta$-grading is twice the homological grading minus the quantum grading. On the knot Floer side, the master complex $\cfk_{U,V}(T_{3,3k+1})$ can be computed from its Alexander polynomial \cite{OS4}, and it is given by
\scriptsize
\[      x_{0} \xrightarrow{U} x_{1} \xleftarrow{V^{2}} x_{2}\xrightarrow{U} ... \xleftarrow{V^{2}}x_{2k-2}  \xrightarrow{U} x_{2k-1} \xleftarrow{V^{2}} x_{2k} \xrightarrow{U^{2}}x_{2k+1} \xleftarrow{V}  x_{2k+2} \xrightarrow{U^{2}}... \xleftarrow{V}x_{4k-2}\xrightarrow{U^{2}}x_{4k-1} \xleftarrow{V}  x_{4k}  \]

\normalsize
\noindent
Substituting $U=V$, $UV=0$ gives 
\[      x_{0} \xrightarrow{U} x_{1} \xleftarrow{0} x_{2}\xrightarrow{U} ... \xleftarrow{0}x_{2k-2}  \xrightarrow{U} x_{2k-1} \xleftarrow{0} x_{2k} \xrightarrow{0}x_{2k+1} \xleftarrow{U}  x_{2k+2} \xrightarrow{0}... \xleftarrow{U}x_{4k-2}\xrightarrow{0}x_{4k-1} \xleftarrow{U}  x_{4k}  \]

\noindent
Taking homology, the Poincare polynomial of $\hfk_{2}(T_{3,3k+1})$ is given by
\begin{equation} \label{eq5.5}
q^{4k-1}(1+3q^{2} + 4q^{4}+4q^{6}+...+4q^{2k}+2q^{2k+2})
\end{equation}

\noindent
Comparing (\ref{eq5.4}) with (\ref{eq5.5}), we see that there is a spectral sequence from $H_{2}(T_{3,3k+1})$ to $\hfk_{2}(T_{3,3k+1})$ where each differential $d_{k}$ is homogeneous of degree 2. 

From the same sources, when $m=3k+2$, the Poincare polynomials of $H_{2}(T_{3,3k+2})$ and $\hfk_{2}(T_{3,3k+2})$ are given by
\[ q^{4k+1}(2+5q^{2} + 6q^{4}+6q^{6}+...+6q^{2k}+3q^{2k+2}) \]

\noindent
and 
\[ q^{4k+1}(2+4q^{2} + 4q^{4}+4q^{6}+...+4q^{2k}+2q^{2k+2}) \]

\noindent
respectively. There is clearly a spectral sequence from the former to the latter where each differential increases $\delta$-grading by $2$.

\end{proof}

\begin{lem}

If $K$ is any knot in $S^{3}$ with crossing number $\le13$, then there is a spectral sequence from $H_{2}(K)$ to $\hfk_{2}(K)$.

\end{lem}

\noindent
This lemma was verified computationally using Bar-Natan's KnotTheory package in Mathematica to compute Khovanov homology and the new C++ program by Ozsv\'{a}th and Szab\'{o} which computes the complex $\cfk_{U,V}(K)/UV=0$.

\subsubsection{$n \ge 3$:} It is difficult to find examples where unreduced Khovanov-Rozansky homology has been computed. For torus knots, there is a known formula for the $(2,m)$ torus knots, and a conjectured formula for $(3,m)$ torus knots. 

\begin{lem} \label{lem5.7}

If $K$ is the $(2,m)$ torus knot, then there is a spectral sequence from $H_{n}(K)$ to $\hfk_{n}(K)$.

\end{lem}

\begin{proof} Let $[n]_{q} = \frac{q^{n}-q^{-n}}{q-q^{-1}}$. In \cite{Cautis2015}, Cautis shows that the Poincare polynomial $\mathcal{P}_{n}$ of $H_{n}(T_{2,2k+1})$ is given by 

\[  \mathcal{P}_{n}(T_{2,2k+1}, q, t) = 1+[n-1]_{q} \Big[ q^{n} \sum_{i=0}^{k} t^{2i}q^{-4i} + t^{3}q^{-n-4}\sum_{i=0}^{k-1}  t^{2i}q^{-4i} \Big] \]

\noindent
The grading $\mathbf{gr}_{n} +\frac{n}{2}\mathbf{gr}_{v}$ corresponds to substituting $t=q^{n}$:
\[ \mathcal{P}_{n}(T_{2,2k+1}, q, q^{n})= 1+[n-1]_{q} \Big[ q^{n} \sum_{i=0}^{k} q^{2i(n-2)} + q^{2n-4}\sum_{i=0}^{k-1} q^{2i(n-2)}  \Big] \]

On the knot Floer side, the master complex $\cfk_{U,V}(T_{2,2k+1})$ is given by 

\[ x_{0} \xrightarrow{U} x_{1} \xleftarrow{V} x_{2} \xrightarrow{U} x_{3} \xleftarrow{V}... \xrightarrow{U} x_{2k-1} \xleftarrow{V} x_{2k} \]

\noindent
Substituting $V=nU^{n-1}$, $UV=0$ gives 
\[ x_{0} \xrightarrow{U} x_{1} \xleftarrow{nU^{n-1}} x_{2} \xrightarrow{U} x_{3} \xleftarrow{nU^{n-1}}... \xrightarrow{U} x_{2k-1} \xleftarrow{nU^{n-1}} x_{2k} \]

\noindent
If we filter the complex by blocking discs which pass through $w$, this complex becomes
\[ x_{0} \xrightarrow{0} x_{1} \xleftarrow{nU^{n-1}} x_{2} \xrightarrow{0} x_{3} \xleftarrow{nU^{n-1}}... \xrightarrow{0} x_{2k-1} \xleftarrow{nU^{n-1}} x_{2k} \]

There is a spectral sequence from the homology of this complex to $\hfk_{2}(T_{2,2k+1})$. But the homology of this complex is given by 
\[ 1+[n-1]_{q} \Big[ q^{n} \sum_{i=0}^{k} q^{2i(n-2)} + q^{2n-4}\sum_{i=0}^{k-1} q^{2i(n-2)}  \Big] \]

\noindent
where the $1$ corresponds to $x_{0}$, the terms in the first sum correspond to $U^{i}x_{j}$ for $i \in \{1,...,n-1\}$, $j$ even, and the terms in the second sum correspond to $U^{i}x_{j}$ for $i \in \{0,...,n-2\}$, $j$ odd. This proves the lemma.

\end{proof}

\begin{lem} \label{lem5.8}

Let $K$ be the $(3,m)$ torus knot, and let $H^{conj}_{n}(K)$ be the homology conjectured by Gorsky and Lewark in \cite{GorskyLewark}  to be the $\sln$ homology of $K$. There is a spectral sequence from ungraded $H^{conj}_{n}(K)$ to ungraded $\hfk_{n}(K)$.

\end{lem}

\begin{proof} The formula for $H^{conj}_{n}(T_{3,3k+1})$ is hard to deal with in full generality, but taking the limit as $q,t \mapsto 1$ yields 
\[ \dim(H^{conj}_{n}(T_{3,3k+1})) = n -2k +4nk + 6k^{2}(N-2)  \]

\noindent
where $n \ge 2$. Note that when $n=2$, this agrees with (\ref{eq5.4}). Recall that the master complex $\cfk_{U,V}(T_{3,3k+1})$ is given by

\small

\[      x_{0} \xrightarrow{U} x_{1} \xleftarrow{V^{2}} x_{2}\xrightarrow{U} ... \xrightarrow{U} x_{2k-1} \xleftarrow{V^{2}} x_{2k} \xrightarrow{U^{2}}x_{2k+1} \xleftarrow{V}  ... \xleftarrow{V}x_{4k-2}\xrightarrow{U^{2}}x_{4k-1} \xleftarrow{V}  x_{4k}  \]

\normalsize

\noindent
Plugging in $V=nU^{n-1}$, $UV=0$ and taking homology yields 

\[ \dim(\hfk_{n}(T_{3,3k+1})) = \begin{cases} 
     2+4k & \textrm{ if $n=2$ } \\
      n+6k &  \textrm{ if $n \ge 3$} \\ 
   \end{cases} \]

\noindent
Thus, $\dim(H_{n}^{conj}(T_{3,3k+1})) \ge \dim(\hfk_{n}(T_{3,3k+1}))$ for all $n \ge 2$, and we have already shown that they agree for $n=1$. In order to complete the proof, we need to show that the dimensions agree modulo 2. However, by inspection, both $\dim(H_{n}^{conj}(T_{3,3k+1}))$ and $\dim(\hfk_{n}(T_{3,3k+1}))$ are congruent to $n$ modulo 2. The proof for $T_{3,3k+2}$ follows from the same argument.

\end{proof}

Assuming the conjectured formulas for the $\sln$ homology of $T_{3,m}$ are correct, this lemma gives some interesting behavior of the conjectured spectral sequences. For 2-bridge knots, in all the examples we have computed, the spectral sequence from $H_{n}(K)$ to $\hfk_{n}(K)$ can be obtained by by the same method as Lemma \ref{lem5.7}: filtering the master complex by the $w$ basepoint. The same is true for 3-bridge knots when $n=2$ - the only known examples of knots with $\rank(\overline{H}_{2}(K)) > \rank(\widehat{\hfk}(K))$ have bridge number at least 4. However, we can see from Lemma \ref{lem5.8} that for $n \ge 3$, there are 3-bridge knots for which the spectral sequence is not induced by any filtration on the knot Floer complex, as for large $k$ the homology $H^{conj}_{n}(T_{3,3k+1})$ has larger rank than the entire complex $\cfk_{n}(T_{3,3k+1})$.

%
%
%\subsection{The $m=1$ Composition Product} The composition product is a formula defined by Jaeger in \cite{Jaeger} which relates sums of HOMFLY-PT polynomials of certain ``sublinks" of $L$ to the HOMFLY-PT polynomial of $L$. This construction was categorified by Wagner in \cite{Wagner} for singular braids, to give a relationship between $H_{m}$, $H_{n}$, and $H_{m+n}$. In this section, we will describe how this construction relates to the conjectured spectral sequences.
%
%\begin{defn}
%
%Let $\mathcal{E}$ denote the edges of a completely singular braid diagram $S$. A \emph{labeling} is a map 
%\[f: \mathcal{E} \to \{1,2\} \]
%
%\noindent
%such that at each vertex $v \in S$, the number of incoming edges labeled with a $1$ is the same as the number of outgoing edges labeled with a $1$. (Note that this forces the same to be true of edges labeled with a $2$.)
%
%\end{defn}
%
%Given a labeling $f$ of $S$, let $S_{f,i}$ denote the diagram $f^{-1}(i)$. Note that $S_{f,i}$ is also a singular braid diagram. 
%
%\begin{lem}[\hspace{1sp}\cite{Wagner}] There is an isomorphism
%
%\[ H_{m+n}(S) \cong \bigoplus_{f} H_{m}(S_{f,1}) \otimes H_{n}(S_{f,2}) \{\sigma_{m,n}(f) \} \]
%
%\noindent
%where $\{\sigma_{m,n}(f) \}$ is a grading shift depending on $m, n, $ and $f$.
%
%\end{lem}
%
%The proof of Lemma \ref{isolemma} required this formula for $m=1$. In other words, we showed that 
%
%\[ H_{F(n)}(S) \cong \bigoplus_{f} H_{1}(S_{f,1}) \otimes H_{n-1}(S_{f,2}) \{\sigma_{m,n}(f) \} \]
%
%\noindent
%On the 

\bibliography{TriplyGradedHomology}{}
\bibliographystyle{plain}

\end{document}